\documentclass[11pt]{amsart}
\usepackage{amsmath,amssymb,mathrsfs}
\usepackage[pdf]{pstricks}
\usepackage[small,heads=LaTeX,nohug]{diagrams}
\usepackage[colorlinks=true,linkcolor=black,citecolor=black,urlcolor=black]{hyperref}

\psset{unit=1pt}
\psset{arrowsize=4pt 1}
\psset{linewidth=.5pt}

\newcommand{\define}{\textbf}

\newcommand{\excise}[1]{}

\newcommand{\isom}{\cong}
\renewcommand{\setminus}{\smallsetminus}
\renewcommand{\phi}{\varphi}

\renewcommand{\tilde}{\widetilde}

\renewcommand{\bar}{\overline}

\renewcommand{\AA}{\mathbb{A}}
\newcommand{\CC}{\mathbb{C}}
\newcommand{\GG}{\mathbb{G}}
\newcommand{\PP}{\mathbb{P}}
\newcommand{\QQ}{\mathbb{Q}}
\newcommand{\RR}{\mathbb{R}}
\newcommand{\ZZ}{\mathbb{Z}}

\newcommand{\shfE}{\mathscr{E}}
\newcommand{\shfF}{\mathscr{F}}

\newcommand{\shfP}{\mathscr{P}}
\newcommand{\shfQ}{\mathscr{Q}}

\newcommand{\catA}{\mathcal{A}}

\newcommand{\catCoh}{\mathbf{Coh}}

\newcommand{\KK}{\mathcal{K}}
\newcommand{\rR}{\mathcal{R}}

\newcommand{\pp}{\mathfrak{p}}

\newcommand{\ee}{\mathrm{e}}  

\newcommand{\kk}{k}    

\newcommand{\OO}{\mathcal{O}}

\newcommand{\id}{\mathrm{id}}
\newcommand{\pt}{\mathrm{pt}}

\newcommand{\opk}{\mathrm{op}K}

\newcommand{\scat}{{\mathbf\Delta}}
\newcommand{\bn}{\mathbf{n}}
\newcommand{\bm}{\mathbf{m}}

\newcommand{\sk}{\mathrm{sk}}
\newcommand{\cosk}{\mathrm{cosk}}

\newcommand{\tor}{T\hspace{-.6ex}or}

\DeclareMathOperator{\Span}{span}

\DeclareMathOperator{\Hom}{Hom}

\DeclareMathOperator{\Proj}{Proj}
\DeclareMathOperator{\Spec}{Spec}

\DeclareMathOperator{\Star}{Star}
\DeclareMathOperator{\Tor}{Tor}
\DeclareMathOperator{\PExp}{PExp}
\DeclareMathOperator{\Exp}{Exp}
\DeclareMathOperator{\Pic}{Pic}

\newtheorem{theorem}{Theorem}[section]
\newtheorem{lemma}[theorem]{Lemma}
\newtheorem{proposition}[theorem]{Proposition}
\newtheorem{corollary}[theorem]{Corollary}

\theoremstyle{definition}
\newtheorem{definition}[theorem]{Definition}
\newtheorem{remark}[theorem]{Remark}
\newtheorem{example}[theorem]{Example}
\newtheorem*{notation}{Notation}

\begin{document}

\title{Operational $K$-theory}

\author{Dave Anderson}
\address{Department of Mathematics\\The Ohio State University\\Columbus, OH 43210}
\email{anderson.2804@math.osu.edu}

\author{Sam Payne}
\address{Department of Mathematics\\Yale University\\New Haven, CT 06511}
\email{sam.payne@yale.edu}
\date{March 9, 2015}
\thanks{DA partially supported by NSF Grant DMS-0902967, the Fondation Sciences Math\'ematiques de Paris (FSMP) (reference: ANR-10-LABX-0098), and a postdoctoral fellowship from the Instituto Nacional de Matem\'atica Pura e Aplicada (IMPA).  SP partially supported by NSF DMS-1068689 and NSF CAREER DMS-1149054.}

\begin{abstract}
We study the operational bivariant theory associated to the covariant theory of Grothendieck groups of coherent sheaves, and prove that it has many geometric properties analogous to those of operational Chow theory.  This operational $K$-theory agrees with Grothendieck groups of vector bundles on smooth varieties, admits a natural map from the Grothen\-dieck group of perfect complexes on general varieties, satisfies descent for Chow envelopes, and is $\AA^1$-homotopy invariant.

Furthermore, we show that the operational $K$-theory of a complete linear variety is dual to the Grothendieck group of coherent sheaves.  As an application, we show that the $K$-theory of perfect complexes on any complete toric threefold surjects onto this group.  Finally we identify the equivariant operational $K$-theory of an arbitrary toric variety with the ring of integral piecewise exponential functions on the associated fan.
\end{abstract}

\maketitle

\setcounter{tocdepth}{1}
\tableofcontents

\section{Introduction}

The Grothendieck groups of vector bundles $K^\circ(X)$ and of coherent sheaves $K_\circ(X)$ are important invariants of a quasi-projective scheme $X$, and each plays a central role in one of the classical formulations of Riemann-Roch theorems. The functor $K_\circ$ is covariant for proper maps, and $K^\circ$ is contravariant for arbitrary maps.\footnote{We adopt the convention, standard in intersection theory but not in $K$-theory, of using superscripts for naturally contravariant functors, and subscripts for covariant functors.}  The two are related by a natural homomorphism $K^\circ(X) \to K_\circ(X)$, which is an isomorphism whenever $X$ is smooth.

As part of a program to unify and strengthen several variants of the Riemann-Roch theorem, Fulton and MacPherson introduced the notion of a \emph{bivariant theory}, which associates a group to each morphism of quasiprojective schemes $X \to Y$, and is equipped various natural operations \cite{bt}.  Their bivariant group
\[
  K^\circ(X \xrightarrow{f} Y)
\]
is the Grothendieck group of $f$-perfect complexes on $X$.\footnote{For a morphism $f$ of quasiprojective schemes, an $f$-perfect complex is a complex of coherent sheaves $\shfP_\bullet$ with the property that, when $f$ is factored as
\[
  X \xrightarrow{\iota} M \xrightarrow{p} Y,
\]
with $\iota$ a closed embedding and $p$ a smooth morphism, $\iota_*\shfP_\bullet$ can be resolved by a finite complex of vector bundles on $M$.}  These groups encompass both the covariant functor
\[
  K_\circ(X) = K^\circ(X \to \pt)
\]
and the contravariant functor
\[
  K^\circ(X) = K^\circ(X \xrightarrow{\id} X),
\]
but possess a great deal more structure, allowing for simplified proofs of the Riemann-Roch theorems.  The corresponding Riemann-Roch theorems were later extended to remove the quasi-projective hypothesis in \cite{fulton-gillet}.

As a bivariant algebraic $K$-theory, however, Grothendieck groups of $f$-perfect complexes are somewhat less than one should hope for.\footnote{A ``bivariant algebraic $K$-theory'' was also defined by Kassel \cite{kassel} (see also \cite{cuntz}, \cite{ct}, and \cite[Ex.~II.2.14]{weibel}).  This is distinct from the bivariant theory studied by Fulton and MacPherson, as it depends on a pair of algebras, but does not involve a map between them.  However, it does include product operations, as well as $K^\circ$ and $K_\circ$ as special cases.}  
The independent squares in this theory---those commuting squares for which one can define a pullback $K^\circ(X \to Y) \to K^\circ(X' \to Y')$---are only the $\Tor$-independent squares, because there is no obvious pullback of an $f$-perfect complex through an arbitrary fiber square \cite[Section~10.8]{bt}.  Furthermore, the contravariant groups $K^\circ(X)$ are difficult to compute on singular spaces.  Even on spaces with mild singularities, such as simplicial projective toric varieties, Grothendieck rings of vector bundles (or perfect complexes) can be uncountable \cite{gubeladze}, and are not $\AA^1$-homotopy invariant \cite{chww}.  Further complications arise on singular spaces that are not $\QQ$-factorial or quasi-projective, where coherent sheaves are not known to have resolutions by vector bundles.  On such spaces, it is not known whether Grothendieck rings of vector bundles and perfect complexes agree \cite{resolution}.  For instance, there are complete, singular, nonprojective toric threefolds, such as \cite[Example~4.13]{vbs}, that have uncountable Grothendieck rings of perfect complexes \cite{gharib-karu}, but are not known to have any nontrivial vector bundles at all.

In this paper, we study basic geometric properties of the \emph{operational bivariant $K$-theory} associated to the covariant theory of Grothendieck groups of coherent sheaves.  Given any covariant homology-like theory, such as $K_\circ$, there is a general construction outlined in \cite[\S8]{bt} of an operational bivariant theory.\footnote{The relation between the covariant and contravariant parts of this operational theory are loosely analogous to the relationship between (contravariant) differential forms and (covariant) currents in differential geometry.}  Roughly speaking, an element of $\opk^\circ(X \rightarrow Y)$ is a collection of operators $K_\circ(Y') \rightarrow K_\circ(X')$, indexed by fiber squares
\[
\begin{diagram}
 X' & \rTo & Y' \\
 \dTo &  & \dTo \\
 X  & \rTo^f   & Y,
\end{diagram}  
\]
that commute with proper pushforward, flat pullback, and Gysin maps for regular embeddings.  A precise definition is given in Section~\ref{s:opk}.  An $f$-perfect complex $\shfP_\bullet$ determines a natural collection of operators, given by
\[
[\shfF] \mapsto \sum_i (-1)^i [\tor_i^Y(\shfP_\bullet,\shfF)], 
\]
for a coherent sheaf $\shfF$ on $Y'$.  See, for instance, \cite[Ex.~18.3.16]{it} or \cite[IV, 2.12]{sga6}.  As explained in Section~\ref{s:gysin}, at least in the case when $f$ is a closed embedding, these operators commute with proper pushforward, flat pullback, and Gysin maps for regular embeddings, giving a natural map from $K^\circ(X \rightarrow Y)$ to $\opk^\circ(X \rightarrow Y)$.

Advantages of passing to the operational theory include the ability to work with arbitrary fiber squares as independent squares, and computability on relatively simple spaces, such as toric varieties.  In future work, we intend to address Grothendieck transformations and Riemann-Roch theorems in this operational framework.

Our general results are given in the setting of separated schemes of finite type over a fixed field.  To describe them more precisely, let $\opk^\circ(X)$ denote the contravariant part of operational $K$-theory, which is the associative ring of operators $\opk^\circ(X \xrightarrow{\id} X)$ corresponding to the identity map on $X$.  Although $\opk^\circ(X)$ has no obvious presentation in terms of generators and relations, we show that it has desirable geometric properties and is computable in many cases of interest.  Furthermore, since the identity is a closed embedding, the construction discussed above gives a natural map from the Grothendieck ring $K^\circ(X)$ of perfect complexes on $X$ to $\opk^\circ(X)$.

\begin{theorem}\label{t:intro-homotopy}
For any scheme $X$, the natural pull back map from $\opk^\circ(X)$ to $\opk^\circ(X \times \AA^1)$ is an isomorphism.
\end{theorem}

\begin{theorem} \label{t:intro-smooth}
For a smooth scheme $X$, the natural maps from $K^\circ(X)$ to $\opk^\circ(X)$ and from $\opk^\circ(X)$ to $K_\circ(X)$ are isomorphisms. 
\end{theorem}

\begin{theorem} \label{t:simple-duality}
For any complete linear variety $X$, the natural map from $\opk^\circ(X)$ to $\Hom(K_\circ(X), \ZZ)$, induced by push forward to a point, is an isomorphism.
\end{theorem}

\noindent Here, the class of linear varieties is the one studied by Totaro in \cite{totaro}.  It contains affine spaces of each dimension, the complement of any linear variety embedded in an affine space, and any variety stratified by linear varieties.  For example, any toric variety or Schubert variety is a linear variety.  One consequence of Theorem~\ref{t:simple-duality} is that $\opk^\circ(X)$ is finitely generated for any complete linear variety. 

The $\AA^1$-homotopy invariance of operational  $K$-theory suggests the potential for interesting connections to Weibel's homotopy $K$-theory $KH^*(X)$, another variation on the $K$-theory of perfect complexes with good geometric properties on singular spaces, which is $\AA^1$-homotopy invariant by construction.  See the original paper \cite{weibelKH}, as well as \cite[\S IV.12]{weibel} and \cite{haesemeyer} for details.  In Section~\ref{s:kimura}, we make one first step toward exploring the relations between these theories.  Let $K^\circ_{\mathrm{naive}}(X)$ denote the Grothendieck group of vector bundles on $X$.  In Corollary~\ref{c.kh} we show that if $X$ has a smooth birational envelope then there is a natural map from the degree zero part of $KH^*(X)$ to $\opk^\circ(X)$, which forms one step in a natural sequence of maps\footnote{If $X$ is smooth, then each of these maps is an isomorphism.}
\[
  K^\circ_{\mathrm{naive}}(X) \to K^\circ(X) \to KH^\circ(X) \to \opk^\circ(X) \to K_\circ(X),
\]
factoring the forgetful map $K^\circ_{\mathrm{naive}}(X) \rightarrow K_\circ(X)$.  In particular, such a map exists for arbitrary varieties over a field of characteristic zero, and for toric varieties over an arbitrary field.  Examples~\ref{ex:nodal}, \ref{ex:normal}, and \ref{ex:toric3} show that $KH^\circ(X) \to \opk^\circ(X)$ is not always injective, even for normal projective toric varieties.

For a toric variety $X$, the natural map $K^\circ(X) \rightarrow KH^\circ(X)$ is split surjective \cite[Proposition~5.6]{chww}.  If the map $KH^\circ(X) \to \opk^\circ(X)$ is also surjective, then we have a surjection $K^\circ(X) \to \opk^\circ(X)$.  (If, in addition, every coherent sheaf on $X$ is a quotient of a vector bundle, as is the case when $X$ is smooth or quasi-projective, then every class in $\opk^\circ(X)$ comes from a difference of vector bundles.)  In Theorem~\ref{t.toric3fold}, we prove that for a three-dimensional toric variety, the map $KH^\circ(X) \to \opk^\circ(X)$ is indeed surjective.  As mentioned earlier, it is not known whether such a variety carries a nontrivial vector bundle.  However, the preceding observations, combined with Theorem~\ref{t:simple-duality}, show that it does have nontrivial perfect complexes:

\begin{theorem}\label{t.toric-nontrivial}
For any complete three-dimensional toric variety $X$ over an algebraically closed field, the map $K^\circ(X) \to \Hom(K_\circ(X),\ZZ)$ is surjective.  In particular, $K^\circ(X)$ is nontrivial.
\end{theorem}

Theorem~\ref{t.toric-nontrivial} can be understood as complementary to results of Gharib and Karu \cite{gharib-karu}.  We find a finitely generated subgroup of nontrivial classes in $K^\circ(X)$, for an arbitrary complete toric threefold by lifting from $KH^\circ(X)$; they find a nontrivial $k$-vector space in the kernel of the map $K^\circ(X) \to KH^\circ(X)$, for many interesting examples of complete toric threefolds.

In the body of the paper, we work equivariantly, with respect to an action of a split torus $T$.  Theorems~\ref{t:intro-smooth} and \ref{t:simple-duality} are the special cases of Theorems~\ref{t:smooth} and \ref{t:duality}, respectively, where the torus is trivial.  Theorem~\ref{t:intro-homotopy} is the special case of Theorem~\ref{t:homotopy} where both the torus and the affine bundle are trivial.  The last of our main results addresses the special case that initially motivated this project, the equivariant $K$-theory of a singular toric variety.

\subsection*{Equivariant $K$-theory of toric varieties}  Throughout this paper, by a {\em toric variety} we mean a normal rational variety, together with a split torus acting with a free open orbit; such a variety corresponds to a fan $\Delta$ as described in \cite{Fulton}.  Let $X$ be a toric variety with dense torus $T$.  The restriction of an equivariant vector bundle to a $T$-fixed point is a representation of $T$, and these representations satisfy a compatibility condition: whenever two fixed points are connected by an invariant curve, the corresponding representations agree on its stabilizer.  If $X$ is smooth and complete, then the induced localization map into a product of copies of the representation ring of the torus,
\[
  K^\circ_T(X) \rightarrow \prod_{x \in X^T} R(T),
\]
is an isomorphism onto the subring consisting of consisting of tuples of virtual representations that satisfy this compatibility condition.  Furthermore, the ordinary Grothendieck ring of vector bundles $K^\circ(X)$ is the quotient of $K^\circ_T(X)$ by the ideal generated by differences of characters, giving a $K$-theoretic analogue of the Stanley-Reisner presentation for the cohomology ring of $X$ \cite{vv}.  These results may be seen as $K$-theory versions of Goresky-Kottwitz-MacPherson (GKM) localization for toric varieties \cite{gkm, knutson-rosu}.  (All such localization theorems build on earlier work of many authors, including the seminal results of Atiyah and Segal \cite{atiyah-segal}.)  These subrings of products of representation rings appearing in the equivariant $K$-theory of smooth toric varieties have the following pleasant interpretation in terms of piecewise exponential functions on the corresponding fan \cite[Section~2.4]{bv}.  

Let $M$ be the character lattice of $T$, so the representation ring $R(T)$ is canonically identified with the group ring $\ZZ[M]$.  Each $u \in M$ may be seen as an integral linear function on the dual space $N_\RR = \Hom(M, \RR)$, and the exponentials of these linear functions are linearly independent.  Therefore, $R(T)$ and $\ZZ[M]$ are naturally identified with $\Exp(N_\RR)$, the ring generated by the exponential functions $e^u$ for $u \in M$.  Elements of $\Exp(N_\RR)$ can be expressed essentially uniquely as finite sums
\[
a_1 e^{u_1} + \cdots + a_r e^{u_r},
\]
with $a_i \in \ZZ$ and $u_i \in M$.  Similarly, when $N'_\RR$ is a rational linear subspace of $N_\RR$, we write $\Exp(N'_\RR)$ for the ring generated by exponentials of linear functions in $M' = M / (N'^\perp_\RR \cap M)$.
\begin{definition}
Let $\Delta$ be a fan in $N_\RR$.  The ring of \define{integral piecewise exponential functions} on $\Delta$ is
\[
\PExp(\Delta) = \left\{ \text{continuous }f \colon |\Delta| \rightarrow \RR \;\; \rule[-7.5pt]{.5pt}{20pt} \;\; f|_\sigma \in \Exp(\Span(\sigma)) \mbox{ for each } \sigma \in \Delta \right\}.
\]
\end{definition}

The identifications above give a canonical isomorphism from $\PExp(\Delta)$ to a subring of a product of representation rings satisfying a natural compatibility condition
\[
\PExp(\Delta) \cong \{ (\rho_\sigma) \in \Pi_{\sigma \in \Delta} R(T_\sigma) \ | \ \rho_\sigma |_{T_\tau} = \rho_\tau \mbox{ whenever } \tau \preceq \sigma \}.
\]
Here, $T_\sigma$ is the pointwise stabilizer of the orbit $O_\sigma$ corresponding to a cone $\sigma \in \Delta$; if $\tau \preceq \sigma$ then $T_\tau$ is a subgroup of $T_\sigma$.  If $\shfE$ is an equivariant vector bundle on $X(\Delta)$, then the induced representations of pointwise stabilizers of orbits satisfy the compatibility condition, and hence give an element of $\PExp(\Delta)$.  In the terminology of \cite{vbs}, this piecewise exponential function is the trace of the exponential of the piecewise linear function on a branched cover of $\Delta$ associated to $\shfE$.  Roughly speaking, this means that it is expressed locally as a sum of exponentials of Chern roots.

The orbits $O_\sigma$ are smooth, so $\opk^\circ_T(O_\sigma)$ is naturally isomorphic to
\[
K^\circ_T(O_\sigma) = R(T_\sigma).
\]
In Section~\ref{s:proof}, we show that the virtual representations associated to an operational $K$-theory class satisfy the compatibility condition, giving a natural map from $\opk^\circ_T(X(\Delta))$ to $\PExp(\Delta)$.

\begin{theorem} \label{t:plp}
Let $X$ be the toric variety with dense torus $T$ corresponding to a fan $\Delta$.  Then the natural map from $\opk^\circ_T(X)$ to $\PExp(\Delta)$ is an isomorphism.
\end{theorem}

\noindent In other words, operational $K$-theory satisfies GKM localization with integer coefficients on arbitrary toric varieties.  

The construction of piecewise exponential functions associated to equivariant vector bundles, described above, induces a natural localization homomorphism from $K^\circ_T(X)$ to $\PExp(\Delta)$ for arbitrary toric varieties $\Delta$.  This map factors through the map from $\opk_T^\circ(X)$, and is an isomorphism when $X$ is smooth, but the kernel and image are not known in general, when $X$ is singular.

\medskip
Our main theorems are closely analogous to well-known results in operational Chow theory.  Notably, Theorems~\ref{t:plp} and \ref{t:duality} are $K$-theoretic versions of \cite[Theorem~1]{chow} and \cite[Theorem~2]{totaro}, respectively.   However, the proofs of the foundational results that make such computations possible in operational $K$-theory are substantially different from those in Chow theory.  See the discussion at the beginning of Section~\ref{s:kimura} for details.

\medskip

To conclude this introduction, we give some examples illustrating the main theorems.

\begin{example}\label{ex.1}
Fix $N=\ZZ^2$, with basis $\{e_1, e_2\}$, and let $\{u_1,u_2\}$ be the dual basis for $M$.  The fan for the weighted projective space $X=\PP(1,1,2)$ has rays through the lattice points $e_1$, $e_2$, and $-e_1-2e_2$.  A sketch of this fan $\Delta$, together with a piecewise exponential function $\xi$, are shown below.

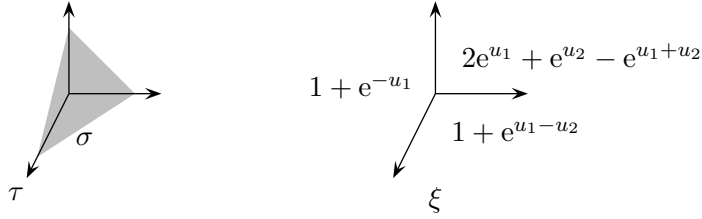
\begin{center}
\begin{pspicture}(-35,-50)(100,50)

\pspolygon*[linecolor=lightgray](0,0)(25,0)(0,25)
\pspolygon*[linecolor=lightgray](0,0)(-12,-24)(0,25)
\pspolygon*[linecolor=lightgray](0,0)(25,0)(-12,-24)
\psline{->}(0,0)(35,0)
\psline{->}(0,0)(0,35)
\psline{->}(0,0)(-16,-32)

\rput(6,-17){$\sigma$}
\rput(-20,-38){$\tau$}

\end{pspicture}
\begin{pspicture}(-35,-50)(100,50)

\psline{->}(0,0)(35,0)
\psline{->}(0,0)(0,35)
\psline{->}(0,0)(-16,-32)

\rput[l](10,14){$2\ee^{u_1}+\ee^{u_2}-\ee^{u_1+u_2}$}
\rput[l](6,-14){$1+\ee^{u_1-u_2}$}
\rput[r](-8,3){$1+\ee^{-u_1}$}

\rput(0,-40){$\xi$}

\end{pspicture}
\end{center}

Let $\sigma$ be the maximal cone spanned by $e_1$ and $-e_1-2e_2$, let $\tau$ be the ray spanned by $-e_1-2e_2$, and take $D=V(\tau)$ and $p=V(\sigma)$ to be the corresponding divisor and fixed point, respectively.  Since $X\setminus D$ and $D\setminus p$ are $T$-invariant affine spaces, the classes $[\OO_X]$, $[\OO_D]$, $[\OO_p]$ form an $R(T)$-module basis for $K^T_\circ(X)$.  Taken together, Theorems~\ref{t:plp} and \ref{t:duality} say that the duals $[\OO_X]^\vee$, $[\OO_D]^\vee$, $[\OO_p]^\vee$ form a basis for $\PExp(\Delta) = \opk_T^\circ(X) = \Hom_{R(T)}(K^T_\circ(X),R(T))$.  Piecewise exponential functions corresponding to this basis are as follows:

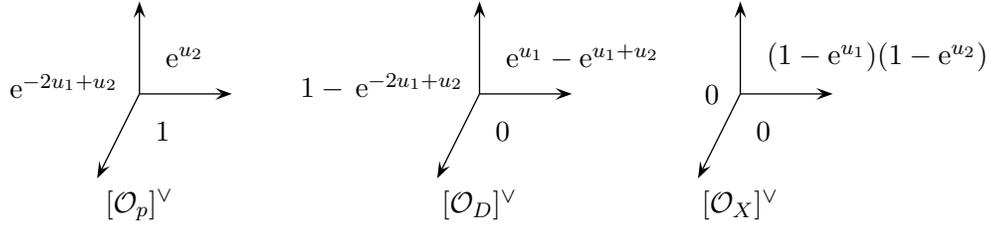
\begin{pspicture}(-35,-50)(80,50)

\psline{->}(0,0)(35,0)
\psline{->}(0,0)(0,35)
\psline{->}(0,0)(-16,-32)

\rput[l](10,14){$\ee^{u_2}$}
\rput[l](6,-14){$1$}
\rput[r](-8,3){$\ee^{-2u_1+u_2}$}

\rput(0,-42){$[\OO_p]^\vee$}

\end{pspicture}
\begin{pspicture}(-45,-50)(60,50)

\psline{->}(0,0)(35,0)
\psline{->}(0,0)(0,35)
\psline{->}(0,0)(-16,-32)

\rput[l](10,14){$\ee^{u_1}-\ee^{u_1+u_2}$}
\rput[l](6,-14){$0$}
\rput[r](-6,3){$1-\,\ee^{-2u_1+u_2}$}

\rput(0,-42){$[\OO_D]^\vee$}

\end{pspicture}
\begin{pspicture}(-35,-50)(80,50)

\psline{->}(0,0)(35,0)
\psline{->}(0,0)(0,35)
\psline{->}(0,0)(-16,-32)

\rput[l](10,14){$(1-\ee^{u_1})(1-\ee^{u_2})$}
\rput[l](6,-14){$0$}
\rput[r](-8,0){$0$}

\rput(0,-42){$[\OO_X]^\vee$}

\end{pspicture}

\noindent
This makes it easy to compute.  For example,
\[
  \xi = (1+\ee^{u_1-u_2})[\OO_p]^\vee + [\OO_D]^\vee .
\]
\end{example}

\begin{remark}
The $K$-theory of complete singular toric varieties is already interesting in the special case of weighted projective spaces, generalizing the above example.  In the non-equivariant setting, the $K$-theory of weighted projective space was studied by Al-Amrani.  He shows that both the $K$-theory of coherent sheaves and the topological complex $K$-theory are free $\ZZ$-modules of rank one more than the dimension \cite{al1,al2}.  In the topological setting, recent work of Harada, Holm, Ray, and Williams identifies the equivariant topological $K$-theory of weighted projective space with the ring of integral piecewise exponential functions, under some divisibility conditions on the weights \cite{hhrw}.
\end{remark}

\begin{example}\label{ex:cusp1}
Consider the cuspidal cubic curve $X = \{y^3-x^2 z = 0\} \subseteq \PP^2$, with $T=\kk^*$ acting by $t\cdot[x,y,z]=[t^3 x, t^2 y, z]$.  Since this is a $T$-linear variety, stratified by the singular point $p=[0,0,1]$ and its complement, Theorem~\ref{t:duality} applies.  In fact, $K^T_\circ(X)$ is freely generated over $R(T)$ by $[\OO_X]$ and $[\OO_p]$, so $\opk_T^\circ(X)$ has the duals of these classes as a basis.

Of course, the same holds non-equivariantly, taking $T$ to be the trivial torus.  This stands in contrast to $K^\circ(X)$, which has the ground field $\kk$ as a direct summand; see Example~\ref{ex:cusp2}.
\end{example}

\bigskip

\noindent \textbf{Acknowledgments.} This collaboration began at a three-week workshop on toric topology at the Hausdorff Institute for Mathematics in Bonn.  We thank the Institute for providing an ideal working environment, and fellow participants T.~Bahri, M.~Franz, M.~Harada, and T.~Holm for many
stimulating discussions.  We are especially grateful to N.~Ray and G.~Williams, from whom we learned about the connection between equivariant vector bundles, equivariant $K$-theory, and piecewise exponential functions.  They have studied this relationship in the topological context \cite{rw}.  

Substantial portions of the work on this project were carried out during subsequent visits to Oberwolfach, the Max Planck Institute in Bonn, and the Institute of Mathematics of Jussieu.  We are grateful for the support and hospitality of these institutions.  

While preparing this manuscript, we have also benefited from many valuable discussions with friends and colleagues.  We especially thank V.~Batyrev, W.~Fulton, J.~Gonz\'alez, L.~Illusie, M.~Kapranov, K.~Karu, M.~Levine, S.~Mitchell, M.~Schlichting, T.~Sch\"urg, V.~Srinivas, B.~Totaro, C.~Weibel, and D.~Wilson.

\section{Background on equivariant $K$-theory}\label{s:eqk}

We begin with a review of basic facts about equivariant $K$-theory.  The foundational details are due to Thomason \cite{thomason,thomason-inv}, and an introductory reference for this material is \cite[\S5]{cg}.  Throughout, we work in the category of separated schemes of finite type over a field $\kk$, equipped with an action of a split torus $T$, which may be trivial.  All morphisms are equivariant with respect to the torus action.  

\subsection{Grothendieck groups}\label{ss.kgps} The Grothendieck group of equivariant coherent sheaves $K_\circ^T(X)$ is generated by classes $[\shfF]$ for each equivariant coherent sheaf $\shfF$ on $X$, subject to the relation $[\shfF]=[\shfF']+[\shfF'']$ for each exact sequence
\[
  0 \to \shfF' \to \shfF \to \shfF'' \to 0.
\]
The functor taking $X$ to $K_\circ^T(X)$ is covariant for equivariant proper maps, with the pushforward defined by $f_*[\shfF] = \sum (-1)^i[ R^if_*\shfF ]$.

The Grothendieck group of equivariant perfect complexes $K^\circ_T(X)$ is contravariant for arbitrary equivariant maps, via derived pullback.  Derived tensor product makes $K_T^\circ(X)$ into a ring, and $K^T_\circ(X)$ into a $K_T^\circ(X)$-module.  When $X$ is quasi-projective, or embeddable in a smooth scheme, or more generally, divisorial, the Grothendieck groups of equivariant vector bundles and equivariant perfect complexes are canonically identified, because in this setting all coherent sheaves admit resolutions by vector bundles \cite[Exp.~III, 2.2.9]{sga6}.  In such cases, since vector bundles are flat, the derived pullback is just the ordinary pullback and the derived tensor product is just the ordinary tensor product.  \emph{A priori}, perfect complexes define a $K$-theory that is better-behaved than the Grothendieck group of vector bundles; for instance, one has localization and Mayer-Vietoris sequences \cite{thomason-trobaugh}.\footnote{It is an open problem whether coherent sheaves admit resolutions by vector bundles on arbitrary separated schemes of finite type over a field.  On non-separated schemes, such as the bug-eyed plane, they do not.  See~\cite{resolution}.}

The ring $K_T^\circ(\pt)$ is isomorphic to the representation ring $R(T)$, so projection to a point makes $K_T^\circ(X)$ into an algebra over $R(T)$.  Letting $M=\Hom(T,\kk^*)$ be the character group, we have a natural isomorphism $R(T)\isom\ZZ[M]$, and given a character $u\in M$, we write $\ee^u\in R(T)$ for the corresponding representation class.

\subsection{Change of groups} Both $K^T_\circ$ and $K_T^\circ$ are functorial for change-of-groups homomorphisms: given $T' \to T$, there are natural maps $K^T_\circ(X) \to K^{T'}_\circ(X)$ and $K_T^\circ(X) \to K_{T'}^\circ(X)$, induced by letting $T'$ act on sheaves through its map to $T$.

\subsection{A forgetful map.} Regarding a vector bundle as a coherent sheaf defines a canonical map
\begin{equation}\label{e:poincare}
  K_T^\circ(X) \to K^T_\circ(X)  
\end{equation}
of $R(T)$-modules.  In general, this map is neither injective nor surjective, but when $X$ is smooth, it is an isomorphism \cite[Corollary~5.8]{thomason}.

\subsection{$\AA^1$-homotopy invariance} For any $T$-equivariant affine bundle $\pi\colon E \rightarrow X$, flat pullback gives a natural isomorphism
\[
  K^T_\circ(X) \cong K^T_\circ(E).
\]
(See \cite[Theorem~4.1]{thomason}, or \cite[\S5.4]{cg}.)  In particular, for any linear $T$-action on $\AA^1$, there is a natural isomorphism 
\[
K_\circ^T(X)  \cong K_\circ^T(X \times \AA^1),
\]
so Grothendieck groups of equivariant coherent sheaves are $\AA^1$-homotopy invariant.

On the other hand, Grothendieck groups of equivariant perfect complexes are $\AA^1$-homotopy invariant for smooth varieties, but not in general.

\begin{example}\label{ex:cusp2}
For the cuspidal plane curve $X = \Spec \kk[x,y]/(y^2-x^3)$, we have $K^\circ(X) = \ZZ \oplus \kk$.  On the other hand, a Mayer-Vietoris calculation shows that $K^\circ(X\times \AA^1)$ contains $\ZZ\oplus \kk[z]$.  
This an instance of the general fact that for one-dimensional schemes, $K^\circ(X) = K^\circ(X\times\AA^1)$ if and only if $X$ is seminormal \cite[I.3.11, II.2.9.1]{weibel}.

Now let $\bar{X} = \Proj \kk[x,y,z]/(y^2z-x^3)$ be the corresponding projective curve.  A similar calculation shows that $K_\circ(\bar{X}) = \ZZ^2 \oplus \kk$ (use \cite[Ex.~II.8.1b or Ex.~II.8.2]{weibel}).  In the case $\kk=\CC$, we see that $K^\circ(\bar{X})$ is uncountable.  (Compare this with Example~\ref{ex:cusp1}, which shows that $\opk^\circ(\bar{X}) \isom \ZZ^2$.)
\end{example}

\subsection{D\'evissage} 
When $X=\Spec A$ is affine, a torus action is the same as an $M$-grading on $A$, and an equivariant coherent sheaf corresponds to an $M$-graded $A$-module \cite[I.4.7.3]{sga3}.  Given such a module $F$ and a character $u\in M$, let $F(u)$ be the same module with shifted grading: $F(u)_v = F_{v-u}$.  In particular, for each $u\in M$, one obtains an equivariant line bundle corresponding to $A(u)$.  In a common abuse of notation, we denote this line bundle $\ee^u$, since it is isomorphic to the pullback of the corresponding representation.  Note that $[F(u)] = \ee^u\cdot[F]$ in $K_\circ^T(X)$.

\begin{lemma}\label{l:devissage}
Let $X$ be a scheme with an action of $T$.  The classes $[\OO_V]$, for $V\subseteq X$ a $T$-invariant subvariety, generate $K^T_\circ(X)$ as a module over $K_T^\circ(\pt) = R(T)$.
\end{lemma}

\begin{proof}
First consider the case where $X=\Spec A$ is affine, so $A$ is an $M$-graded ring, and an equivariant coherent sheaf corresponds to an $M$-graded $A$-module $F$.  One can find a chain
\[
  0 = F_0 \subset F_1 \subset \cdots \subset F_n = F
\]
of $M$-graded submodules such that $F_i/F_{i-1} \isom (A/\pp_i)(u_i)$, for some $M$-graded prime ideals $\pp_i \subset A$ and elements $u_i \in M$.  (Cf. \cite[Ch.~IV, \S1, Th\'eor\`eme~1]{bourbaki} for the ungraded case.)  
It follows that $[F] = \ee^{u_1}[\OO_{V_1}]+\cdots+\ee^{u_n}[\OO_{V_n}]$, where $V_i = \Spec(A/\pp_i)$.

For an arbitrary $X$, let $U\subseteq X$ be a nonempty $T$-invariant affine open, and let $Y=X\setminus U$.  (Such a $U$ exists, e.g., by applying \cite[Corollary~2]{sumihiro} to the normal locus of $X_{\mathrm{red}}$.)  There is an exact sequence
\begin{equation*}
  K^T_\circ(Y) \to K^T_\circ(X) \to K^T_\circ(U) \to 0.
\end{equation*}
We know $K^T_\circ(U)$ is generated by classes of structure sheaves of subvarieties, by the affine case, and we may assume the lemma for $K^T_\circ(Y)$ by induction on dimension and the number of irreducible components.  It follows that $K^T_\circ(X)$ is also generated by structure sheaves of subvarieties.
\end{proof}

\subsection{When a subtorus acts trivially}  In order to compute effectively in Sections~\ref{s:kimura} and \ref{s:proof}, we will need to handle the case where a subtorus acts trivially on $X$, as is the case for the action of the dense torus on a proper closed $T$-invariant subvariety of a toric variety.  

\begin{lemma}\label{l:Ktrivial}
Suppose $T=T_1\times T_2$ acts on $X$ such that the action of $T_1$ is trivial.  Then there is a canonical isomorphism $K^T_\circ(X) = R(T_1)\otimes K^{T_2}_\circ(X)$.
\end{lemma}

\noindent
The statement seems to be known, but we include an easy proof.  (The argument also shows the same is true for $K^T_i(X)$ with $i>0$.)

\begin{proof}
When $X$ is affine, we have $K^T_i(X) = R(T_1)\otimes K^{T_2}_i(X)$ for all $i\geq 0$ by \cite[Lemma~5.6]{thomason-inv}.  In general, let $U = \Spec A$ be a nonempty $T$-invariant affine open in $X$, and let $Y = X \setminus U$.  We have a diagram

{\footnotesize
\begin{diagram}
K^T_1(U) & \rTo & K^T_\circ(Y) & \rTo & K^T_\circ(X) & \rTo & K^T_\circ(U) & \rTo & 0 \\
\uTo     &      & \uTo         &      &  \uTo        &      &  \uTo  \\
R(T_1)\otimes K^{T_2}_1(U) & \rTo & R(T_1) \otimes K^{T_2}_\circ(Y) & \rTo & R(T_1) \otimes K^{T_2}_\circ(X) & \rTo & R(T_1) \otimes K^{T_2}_\circ(U) & \rTo & 0
\end{diagram}
}

\noindent
in which the rows are exact, the first and fourth vertical arrows are isomorphisms by the affine case, and the second vertical arrow is an isomorphism by induction on dimension and number of irreducible components.  An application of the five lemma completes the proof.
\end{proof}

\subsection{Gillet's exact sequence for envelopes} We will make essential use of the following equivariant analogue of a result of Gillet.  An \define{equivariant envelope} is a proper equivariant map $f\colon X' \to X$ such that for every $T$-invariant subvariety $V \subseteq X$, there is an invariant subvariety $V'\subseteq X'$ mapping birationally onto $V$.  

\begin{proposition}\label{p:env-exact}
Suppose $f\colon X' \to X$ is an equivariant envelope.  Then the sequence
\[
K^T_\circ(X'\times_X X') \xrightarrow{p_{1*}-p_{2*}} K^T_\circ(X') \xrightarrow{f_*} K^T_\circ(X) \to 0
\] 
is exact, where $p_1,p_2\colon X'\times_X X' \to X'$ are the projections.
\end{proposition}

\noindent The non-equivariant version can be found in \cite[p.~300]{fulton-gillet} and \cite[Corollary~4.4]{gillet}.  While surjectivity of $f_*$ follows easily from Lemma~\ref{l:devissage}, exactness in the middle seems to a require a more complicated argument. The main ingredients of the proof are a descent theorem and a spectral sequence for equivariant $K$-theory of simplicial schemes; we will give a more detailed discussion in the appendix.  Exactness of the corresponding sequence for equivariant Chow groups is more elementary (see \cite[Theorem~1.8]{kimura} and \cite[\S2]{chow}).

\section{Refined Gysin maps}\label{s:gysin}

Consider a fiber square
\begin{equation}\label{e.fiber}
\begin{diagram}
 X' & \rTo^{f'} & Y' \\
 \dTo^{g'} &  & \dTo_g \\
 X  & \rTo^f   & Y.
\end{diagram}  
\end{equation}
As mentioned in the introduction, an $f$-perfect complex $\shfP_\bullet$ determines an operator $f^{\shfP_\bullet}: K_\circ(Y') \rightarrow K_\circ(X')$, given by
\[
[\shfF] \mapsto \sum_i (-1)^i \big[ Tor^Y_i (\shfP_\bullet, \shfF) \big].
\]
When $f$ has finite $\Tor$-dimension, which means that $\OO_X$ itself is $f$-perfect, we write $f^!$ for $f^{\OO_X}$, and call this the \emph{refined Gysin map} for $f$.  These are closely analogous to the refined Gysin maps for local complete intersection morphisms in Chow theory  \cite[\S6.2]{it}.  Here we review the necessary basic facts about $\Tor$ sheaves and refined Gysin maps in $K$-theory, following  \cite[Exp. III]{sga6}, \cite[III.6]{ega}, \cite[VI.6]{fl}, and especially \cite[\S2.2]{ks}.

\subsection{Tor sheaves} Recall that, given a sheaf $\shfE$ on $X$ and a sheaf $\shfF$ on $Y'$, one has Tor sheaves $\tor^Y_i(\shfE,\shfF)$ supported on $X'$.  In the affine case, writing $X=\Spec A$, $Y=\Spec B$, and $Y'=\Spec B'$, with $E$ an $A$-module and $F$ a $B'$-module, these are the sheaves defined by the $A\otimes_B B'$-modules $\Tor^B_i(E,F)$.  In the general case, one covers $X'$ by affines of the form $\Spec(A\otimes_B B')$, and takes the associated sheaves.  The map $f$ has \emph{finite Tor-dimension} if, for every sheaf $\shfF$ of $\OO_Y$-modules, all but finitely many of the Tor sheaves $\tor^Y_i(\OO_X,\shfF)$ are zero.

\subsection{Equivariant Tor} When $\shfE$ and $\shfF$ have an equivariant structure, the sheaves $\tor^Y_i(\shfE,\shfF)$ inherit a canonical equivariant structure.  To see this, first observe that the formation of Tor sheaves commutes with flat base change: suppose
\begin{diagram}
 W' & \rTo & Z' \\
 \dTo &  & \dTo \\
 W  & \rTo   & Z,
\end{diagram}  
is a fiber square, with flat maps $u\colon Z \to Y$, $u'\colon Z' \to Y'$, $v\colon W \to X$, $v'\colon W' \to X'$ such that each face of the cube, with top and bottom faces formed from this diagram and \eqref{e.fiber} above, is a fiber square.  Then there is a natural isomorphism
\begin{equation}\label{e.tor-natural}
  \tor^Z_i(v^*\shfE,u'^*\shfF) \isom v'^*\tor^Y_i(\shfE,\shfF).
\end{equation}
Take $Z=T\times Y$, $W=T\times X$, etc., let $p,a\colon T\times Y \to Y$ be the projection and action maps, respectively, and define $q,b\colon T\times X \to X$ similarly, as well as $p',a',q',b'$.  Then naturality of the isomorphism \eqref{e.tor-natural} means the isomorphism
\[
  \tor^{T\times Y}_i(q^*\shfE,p'^*\shfF) \isom \tor^{T\times Y}_i(b^*\shfE,a'^*\shfF)
\]
coming from the equivariant structures of $\shfE$ and $\shfF$ induces an isomorphism
\[
  q'^*\tor^Y_i(\shfE,\shfF) \isom b'^*\tor^Y_i(\shfE,\shfF),
\]
giving the equivariant structure of the Tor sheaf.

\subsection{Relatively perfect complexes}

Given a morphism $f\colon X \to Y$, a bounded complex $\shfP_\bullet$ of coherent $\OO_X$-modules is \emph{$f$-perfect} if, for all sheaves $\shfF$ of $\OO_Y$-modules, the Tor sheaves $\Tor^Y_i(\shfP_\bullet,\shfF)$ are zero for all but finitely many $i$.  When $f$ is equivariant, an \emph{equivariant $f$-perfect complex} is simply an equivariant complex that is also $f$-perfect.

There are several equivalent characterizations.  A useful fact is that when $f$ factors as a closed embedding $\iota\colon X \hookrightarrow M$ followed by a smooth projection $p\colon M \to Y$, a complex $\shfP_\bullet$ is $f$-perfect if and only if $\iota_*\shfP_\bullet$ is perfect on $M$, in the absolute sense defined in \S\ref{ss.kgps} above \cite[Exp.~III, 4.4]{sga6}.  Such a factorization exists whenever $X$ is quasi-projective.

\subsection{Equivariant refined Gysin maps}

Because $\Tor$ sheaves carry a canonical equivariant structure, the usual construction of refined Gysin maps in $K$-theory works equivariantly.  With notation as in \eqref{e.fiber}, for an equivariant map $f\colon X \to Y$ and an equivariant $f$-perfect complex $\shfP_\bullet$ on $X$, there is a pullback map $f^{\shfP_\bullet} \colon K_\circ^T(Y') \rightarrow K_\circ^T(X')$, given by
\begin{equation}\label{e.gysindef}
  f^{\shfP_\bullet}[\shfF] = \sum_i (-1)^i [\tor_i^Y(\shfP_\bullet,\shfF)].
\end{equation}
See \cite[\S2.2]{ks} for the non-equivariant case.  As before, we write $f^!$ for $f^{\OO_X}$, when $f$ has finite $\Tor$-dimension.  

In fact, the pullback $f^{\shfP_\bullet}$ only depends on the $K$-class of $\shfP_\bullet$: given a short exact sequence
\[
  0 \to \shfP'_\bullet \to \shfP_\bullet \to \shfP''_\bullet \to 0
\]
of equivariant $f$-perfect complexes, and an equivariant coherent sheaf $\shfF$ on $Y'$, we have 
\[
  f^{\shfP_\bullet}[\shfF] = f^{\shfP'_\bullet}[\shfF] + f^{\shfP''_\bullet}[\shfF]
\]
in $K^T_\circ(X')$.  (This follows from the long exact sequence for $\Tor$.)

\subsection{Commutativity properties of Gysin maps}  

The main facts we will need say that the equivariant Gysin maps associated to flat morphisms and regular embeddings commute with proper push forward and each other.  These are the analogues of the corresponding statements for Chow groups \cite[Proposition~6.3, Theorem~6.4]{it}.  
In fact, we will show something slightly more general.  First consider a diagram of fiber squares
\begin{equation}\label{e.2sq}
\begin{diagram}
 X'' & \rTo & Y''  \\
\dTo^{h'}  & & \dTo_h   \\
 X' & \rTo^{f'} & Y'  \\
\dTo^{g'} & & \dTo_g \\
X  & \rTo^{f}  & Y .
\end{diagram}
\end{equation}

\begin{lemma}\label{l.push-commute}
In the diagram of fiber squares \eqref{e.2sq}, suppose $h$ is proper and $f$ is either flat or a closed embedding.  Let $\shfP_\bullet$ be an equivariant $f$-perfect complex on $X$. Then $f^{\shfP_\bullet} h_* = h'_* f^{\shfP_\bullet}$.
\end{lemma}

\noindent Here, the equality is an identity of maps from $K_\circ^T(Y'')$ to $K_\circ^T(X')$.  When $f$ is a closed embedding and $\shfP_\bullet =\OO_X$, the hypothesis says that $f$ is a regular embedding.

\begin{proof}
The proof is similar to a reduction argument given in the proof of \cite[Proposition~2.2.2]{ks}.  The case where $f$ is flat is easy, and left to the reader.  We assume that $f$ is a closed embedding.  Recall that the class of a complex $[\shfF_\bullet]$ of equivariant sheaves on $X$ is defined to be the alternating sum $\sum (-1)^i [\shfF_i]$.  Now let $\shfF$ be an equivariant coherent sheaf on $Y''$.  By \cite[Lemma~1.5.3]{ks}, we have a natural isomorphism
\begin{align}
  \tor^Y_p(\shfP_\bullet,Rh_*\shfF) &\xrightarrow{\sim} R^{-p}h_*(\shfP_\bullet\otimes^L_{\OO_Y} \shfF) \label{e.isom1}
\end{align}
of $\OO_{X'}$-modules, and $E^2$ spectral sequences
\begin{align*}
  \tor^Y_p(\shfP_\bullet, R^{-q}h_*\shfF) &\Rightarrow \tor^Y_{p+q}(\shfP_\bullet, Rh_*\shfF) \\
  \intertext{and}
  R^{-p}h_*\tor^Y_q(\shfP_\bullet,\shfF) &\Rightarrow  R^{-p-q}h_*(\shfP_\bullet\otimes^L_{\OO_Y} \shfF),
\end{align*}
also of $\OO_{X'}$-modules.  The isomorphism \eqref{e.isom1} says that the two spectral sequences converge to the same thing, and so
\begin{align*}
f^{\shfP_\bullet}h_*[\shfF] &= \sum_{p,q} (-1)^{p+q} [\tor^Y_q(\shfP_\bullet, R^p h_*\shfF)] \\
                  &= \sum_{p,q} (-1)^{p+q} [R^p h'_*\tor^Y_q(\shfP_\bullet,\shfF)] \\
                  &= h'_*f^{\shfP_\bullet}[\shfF],
\end{align*}
using the fact that $X'\hookrightarrow Y'$ is a closed embedding to replace $R^p h_*$ with $R^p h'_*$ in the second line.
\end{proof}

Next, consider the following diagram of fiber squares:

\begin{equation}\label{e.3sq}
\begin{diagram}
 X'' & \rTo & Y'' & \rTo  &  Z''  \\
\dTo & & \dTo    &       & \dTo_{h} \\
 X' & \rTo & Y'    & \rTo  &  Z' \\
\dTo & & \dTo \\
X  & \rTo^{f}  & Y .
\end{diagram}
\end{equation}

\begin{lemma}\label{l.commutes}
In the diagram of fiber squares (\ref{e.3sq}), suppose each of the maps $f$ and $h$ is either flat or a closed embedding.  Let $\shfP_\bullet$ be an equivariant $f$-perfect complex on $X$, and let $\shfQ_\bullet$ be an equivariant $h$-perfect complex on $Z''$.  Then, for any class $\xi$ in $K^T_\circ(Y')$,
\[
  f^{\shfP_\bullet} h^{\shfQ_\bullet}(\xi) = h^{\shfQ_\bullet} f^{\shfP_\bullet}(\xi)
\]
in $K^T_\circ(X'')$.
\end{lemma}

\noindent In other words, $f^{\shfP_\bullet} h^{\shfQ_\bullet} = h^{\shfQ_\bullet} f^{\shfP_\bullet}$ as maps from $K_\circ^T(Y')$ to $K_\circ^T(X'')$.  When $\shfP_\bullet = \OO_X$ and $\shfQ_\bullet = \OO_{Z''}$, the lemma says the refined Gysin maps commute, i.e., $f^! h^! = h^! f^!$.

\begin{proof}
For an equivariant coherent sheaf $\shfF$ on $Y'$, we need to show
\begin{align}\label{e.bigsum}
\begin{aligned}
 & \sum_p (-1)^p \sum_q (-1)^q [\tor^{Z'}_q(\shfQ_\bullet,\tor^Y_p(\shfP_\bullet,\shfF))] \\ &\qquad = \sum_q (-1)^q \sum_p (-1)^p [\tor^Y_p(\shfP_\bullet,\tor^{Z'}_q(\shfQ_\bullet,\shfF))] .
\end{aligned}
\end{align}

This is similar to Lemma~\ref{l.push-commute}.  It is easy when either $f$ or $h$ is flat, so we assume they are both closed embeddings.  In this case, it follows from the natural isomorphism
\begin{align*}
 \tor^Y_q( \shfP_\bullet, \shfQ_\bullet \otimes^L_{Z'} \shfF ) &\xrightarrow{\sim} \tor^{Z'}_q( \shfQ_\bullet, \shfP_\bullet \otimes^L_Y \shfF )
\end{align*}
and the spectral sequences
\begin{align*}
 {'E}^2_{pq} = \tor^{Y}_p( \shfP_\bullet, \tor^{Z'}_q( \shfQ_\bullet, \shfF ) ) & \Rightarrow \tor^Y_{p+q}( \shfP_\bullet, \shfQ_\bullet \otimes^L_{Z'} \shfF ) \\
 \intertext{and}
 {''E}^2_{pq} = \tor^{Z'}_p( \shfQ_\bullet, \tor^Y_q( \shfP_\bullet, \shfF ) ) & \Rightarrow \tor^{Z'}_{p+q}( \shfQ_\bullet, \shfP_\bullet \otimes^L_Y \shfF )
\end{align*}
of $\OO_{X''}$-modules \cite[Lemma~1.5.2]{ks}.
\end{proof}

\subsection{Functoriality}

If $f\colon X \to Y$ and $g\colon Y \to Z$ are morphisms of finite $\Tor$-dimension, then $h = g\circ f\colon X \to Z$ also has finite $\Tor$-dimension.  In general, however, it is not clear that $h^! = f^! \circ g^!$.  In a special case, however, this functoriality is easy to see.

\begin{lemma}\label{l.functorial}
Suppose $f\colon X \to Y$ factors as a closed embedding $\iota\colon X \hookrightarrow M$ followed by a smooth projection $p\colon M \to Y$.  Let $\shfP_\bullet$ be an equivariant $f$-perfect complex on $X$.  Then $f^{\shfP_\bullet} = \iota^{\shfP_\bullet} \circ p^!$ as maps $K^T_\circ(Y') \to K^T_\circ(X')$, for any $Y' \to Y$.
\end{lemma}

\begin{proof}
Given $Y' \to Y$, let $X' \hookrightarrow M' \xrightarrow{p'} Y'$ be the corresponding factorization of the map $f'\colon X' \to Y'$.  Given an equivariant coherent sheaf $\shfF$ on $Y'$, we need to prove
\[
  \sum_j (-1)^j [\tor^Y_j(\shfP_\bullet,\shfF)] = \sum_j (-1)^j [\tor^M_j(\shfP_\bullet,p'^*\shfF)],
\]
since the left-hand side is $f^{\shfP_\bullet}[\shfF]$ and the right-hand side is $\iota^{\shfP_\bullet}(p^![\shfF])$.  This follows from the isomorphism $\tor^M_j(\shfP_\bullet,p'^*\shfF) \xrightarrow{\sim} \tor^Y_j(\shfP_\bullet,\shfF)$ of $\OO_X$-modules \cite[Lemma 1.5.1]{ks}.
\end{proof}

\section{Operational bivariant $K$-theory}\label{s:opk}

Following the general construction of Fulton and MacPherson, which associates an operational bivariant theory to a covariant theory with a distinguished class of commuting Gysin maps, we define operational equivariant $K$-theory, as follows.  Our distinguished Gysin maps are those associated to flat morphisms and regular embeddings.  The independent squares in this bivariant theory are arbitrary fiber squares. 

\smallskip

 Let $f\colon X \rightarrow Y$ be a morphism.

\begin{notation}
For any morphism $Y' \rightarrow Y$, we write $X'$ for the fiber product $X \times_Y Y'$.  
\end{notation}

\begin{definition}
The \define{operational equivariant $K$-theory} is the bivariant group $\opk^\circ_T(f\colon X \rightarrow Y)$ whose elements are collections of operators
\[
c_g \colon K_\circ^T(Y') \rightarrow K_\circ^T(X'),
\]
indexed by morphisms $g\colon Y' \rightarrow Y$, that satisfy the following bivariant axioms.
\renewcommand{\labelenumi}{(A\arabic{enumi})}
\begin{enumerate}
\item \label{axiom:pushforward} In the diagram of fiber squares
\begin{diagram}
 X'' & \rTo^{f''} & Y'' \\
\dTo^{h'} & & \dTo_h \\
 X' & \rTo^{f'} & Y' \\
\dTo^{g'} & & \dTo_g \\
X  & \rTo^{f}  & Y ,
\end{diagram}
if $h$ is proper then $c_g \circ h_* = h'_* c_{gh}$.
\smallskip
\item \label{axiom:gysin} In the diagram of fiber squares
\begin{diagram}
 X'' & \rTo & Y'' & \rTo  &  Z''  \\
\dTo & & \dTo_{h'}    &       & \dTo_{h} \\
 X' & \rTo & Y'    & \rTo  &  Z' \\
\dTo & & \dTo \\
X  & \rTo^{f}  & Y,
\end{diagram}
if $h$ is either flat or a regular embedding then $c_{gh'} \circ h^! = h^! \circ c_g$.
\end{enumerate}
\end{definition}

\noindent Roughly speaking, axioms (A\ref{axiom:pushforward}) and (A\ref{axiom:gysin}) are commutativity with proper push forward and refined Gysin maps, respectively.  Commutativity with flat pullback is the special case of (A\ref{axiom:gysin}) where $h$ is flat, $Y' = Z'$, and $Y' \rightarrow Z'$ is the identity.  The group law on $\opk^\circ_T(f\colon X \rightarrow Y)$ is given by addition of homomorphisms.  

\begin{notation}
When no confusion seems possible, we omit the name of the morphism $f$ and write simply $\opk^\circ_T(X \rightarrow Y)$.

We write $\opk_T^\circ(X)$ for the operational $K$-group of the identity morphism on $X$.  This is an associative ring with unit, by composition of endomorphisms, and it is contravariantly functorial in $X$.
\end{notation}

\subsection{Operations}

We now describe several natural operations on operational $K$-groups.  The first is a pullback that makes $\opk^\circ_T$ into a contravariant functor on the category of morphisms, in which arrows are given by fiber squares.  Let $g\colon Y' \rightarrow Y$ be a morphism, and consider the diagram of fiber squares above.

\renewcommand{\labelenumi}{(O\arabic{enumi})}
\begin{enumerate}
\item The pullback
\[
g^*\colon \opk^\circ_T(X \rightarrow Y) \rightarrow \opk^\circ_T(X' \rightarrow Y')
\]
is given by the collection of operators $(g^*c)_h = c_{gh}$.
\end{enumerate}  

\noindent The next two operations are defined in terms of the following diagram of fiber squares
\begin{diagram}
 X' & \rTo^{f'} & Y' & \rTo^{g'} & Z' \\
\dTo & & \dTo_{h'}   &     & \dTo_{h} \\
 X & \rTo^{f} & Y   &  \rTo^g  &  Z.
\end{diagram}

\noindent First we describe the product operation.

\begin{enumerate}
\setcounter{enumi}{1}
\item The product
\[
\cdot : \opk^\circ_T(X \rightarrow Y) \otimes \opk^\circ_T(Y \rightarrow Z) \rightarrow \opk^\circ_T(X \rightarrow Z)
\]
is given by composition of homomorphisms, so $(c' \cdot c)_h = c'_{h'} \circ c_h$.
\end{enumerate}

\noindent Next, we describe the push forward operation for proper morphisms.

\begin{enumerate}
\setcounter{enumi}{2}
\item If $f$ is proper then the pushforward
\[
f_* \colon \opk^\circ_T(X \rightarrow Z) \rightarrow \opk^\circ_T(Y \rightarrow Z)
\]
is given by $f_*(c)_h = f'_* \circ c_h$.
\end{enumerate}

\subsection{Basic properties} These operations satisfy a number of basic compatibility properties that follow from associativity of composition of operators, standard properties of proper pushforward, and the axioms above.  These are similar to standard properties of operational Chow cohomology \cite[Section~17.2]{it}.  Taken together, these properties verify that $\opk_T^\circ$ satisfies the axioms for a bivariant theory, as defined in \cite[\S2 and \S8]{bt}.

\begin{enumerate}
\renewcommand{\labelenumi}{(P\arabic{enumi})}
\item \emph{Associativity of products.} If $c \in \opk^\circ_T(X\rightarrow Y)$, $d \in \opk^\circ_T(Y \rightarrow Z)$, and $e \in \opk^\circ_T(Z \rightarrow W)$, then
\[
(c \cdot d) \cdot e = c \cdot (d \cdot e).
\]
\item \emph{Functoriality of push forward.} If $f: X\rightarrow Y$ and $g: Y \rightarrow Z$ are proper, with $Z \rightarrow W$ arbitrary, and $c \in \opk^\circ_T(X \rightarrow W)$, then
\[
(gf)_* (c) = g_*(f_*(c)).
\]
\item \emph{Functoriality of pull back.}
If $g: Y' \rightarrow Y$ and $h: Y'' \rightarrow Y'$, are arbitrary, $X'' = X \times_Y Y''$, and $c \in \opk^\circ_T(X \rightarrow Y)$ then
\[
(gh)^* (c) = g^*(h^*(c)).
\]
\item \emph{Product and push forward commute.}
If $f: X \rightarrow Y$ is proper, $Y \rightarrow Z$ and $Z \rightarrow W$ are arbitrary, $c \in \opk^\circ_T(X \rightarrow Z)$, and $d \in \opk^\circ_T(Z \rightarrow W)$ then
\[
f_*(c) \cdot d = f_*( c \cdot d).
\]
\item \emph{Product and pull back commute.} If $c \in \opk^\circ_T(X \rightarrow Y)$, $d \in \opk^\circ_T(Y \rightarrow Z)$, and $h : Z' \rightarrow Z$, with
\begin{diagram}
 X' & \rTo & Y' & \rTo & Z' \\
\dTo & & \dTo_{h'}   &     & \dTo_{h} \\
 X & \rTo & Y   &  \rTo  &  Z
\end{diagram}
the resulting diagram of fiber squares, then
\[
h^*(c \cdot d) = h'^*(c) \cdot h^*(d).
\]
\item \emph{Push forward and pull back commute.}  If $f\colon X \rightarrow Y$ is proper, $Y \rightarrow Z$ and $h\colon Z' \rightarrow Z$ are arbitrary, and $c \in \opk^\circ_T(X \rightarrow Z)$, with
\begin{diagram}
 X' & \rTo^{f'} & Y' & \rTo & Z' \\
\dTo & & \dTo_{h'}   &     & \dTo_{h} \\
 X & \rTo^{f} & Y   &  \rTo  &  Z
\end{diagram}
the resulting diagram of fiber squares, then
\[
h^*(f_*(c)) = f'_*(h^*(c)).
\]
\item \emph{Projection formula.} If $X \rightarrow Y$ and $Y \rightarrow Z$ are arbitrary, and $h' \colon Y' \rightarrow Y$ is proper, then defining notation by the diagram
\begin{diagram}
 X' & \rTo & Y' &  & \\
\dTo^{h''} & & \dTo_{h'}   &     &  \\
 X & \rTo & Y   &  \rTo  &  Z,
\end{diagram}
for $c \in \opk_T^\circ(X \rightarrow Y)$ and $d \in \opk^\circ_T(Y \rightarrow Z)$ we have
\[
c \cdot h'_*(d) = h''_*(h'^*(c) \cdot d).
\]
\end{enumerate}

These properties all follow from the commutativity lemmas of Section~\ref{s:gysin}.  A further property, particular to the equivariant case, is that each group $\opk_T^\circ(X \to Y)$ is an $R(T)$-module.  (Since $R(T) = K_T^\circ(\pt)$, this follows from commutativity with flat pullback (Lemma~\ref{l.commutes}).)

\subsection{Orientations}

By Lemmas~\ref{l.push-commute} and \ref{l.commutes}, the refined Gysin maps associated to flat morphisms and regular embeddings are compatible with proper push forward and commute with each other.  Therefore, if $h \colon X \rightarrow Y$ is such a morphism, the refined pullback $h^!$ defines an element of $\opk^\circ_T(X \rightarrow Y)$, called the \emph{canonical orientation} of $h$, which we denote by $[h]$.

Given morphisms $X \xrightarrow{f} Y \xrightarrow{g} Z$, with $h=g\circ f$, it is not clear that $h^! = f^!g^!$ even when these pullbacks are defined, so in general we do not know if $[h] = [f]\cdot [g]$.  However, in one important special case this property is easy to check.

\begin{lemma}\label{l.sections}
Let $g\colon Y \to Z$ be a smooth morphism, let $f\colon X \to Y$ be a regular embedding, and let $h = g\circ f\colon X \to Z$.  Then given any $Z' \to Z$, we have $f^! g^! = h^!$ as homomorphisms $K^T_\circ(Z') \to K^T_\circ(X')$.  Therefore
\[
[f]\cdot[g] = [h]
\]
in $\opk_T^\circ(X \to Z)$.
\end{lemma}

The proof is immediate from Lemma~\ref{l.functorial}.  As a particular case, when $Z=X$ and $f$ is a section of $g$, so $h=\id$, we see that $[f]\cdot [g] = 1$ in $\opk_T^\circ(X)$.

\subsection{Geometric properties}

These operational $K$-groups also have geometric properties that are similar to those of operational Chow cohomology, but are not immediate formal consequences of the axioms.  The most important of these, for our purposes, is the following proposition, whose proof is similar to that of \cite[Proposition~17.4.2]{it}.

\begin{proposition}\label{p:opk-smooth}
Let $X \to Y$ be an arbitrary morphism, and let $g\colon Y \to Z$ be smooth.  Then
\[
  \cdot [g] \colon \opk_T^\circ(X \to Y) \to \opk_T^\circ(X\to Z)
\]
is an isomorphism.
\end{proposition}

\begin{proof}
Let $f$ be the morphism from $X$ to $Y$.  Form the diagram of fiber squares
\begin{diagram}
 X & \rTo^{f} & Y   \\
\dTo^{\gamma} & & \dTo_{\delta}  \\
 X \times_Z Y & \rTo^{f'} & Y \times_Z Y   &  \rTo^q  &  Y \\
\dTo^{p'}   &               & \dTo_p    &           & \dTo_g \\
X  & \rTo^{f}  & Y      &  \rTo^g        &  Z ,
\end{diagram}
where $\gamma$ is the graph of $f$, $p$ and $q$ are first and second projections, respectively, and $\delta$ is the diagonal, which is a regular embedding because $Y$ is smooth over $Z$.

Define $L\colon \opk^\circ_T(X \rightarrow Z) \rightarrow \opk^\circ_T(X \rightarrow Y)$ by
\[
L (c) = [\gamma] \cdot g^*(c).
\]
We claim that $L$ is a two-sided inverse for $\cdot [g]$.  To prove the claim, we first compute, for $c \in \opk^\circ_T(X \rightarrow Y)$,
\begin{align*}
L(c \cdot [g]) & = [\gamma] \cdot g^*(c \cdot [g]) \\
 & =  [\gamma] \cdot p^*c \cdot g^*[g]  & &\text{(P5), product and pullback commute} \\
 & =  [\gamma] \cdot p^*c \cdot [q] & & \text{(A2), for flat pullbacks} \\
 & =  \delta^* p^*c \cdot [\delta] \cdot [q] && \text{(A2), for regular embeddings} \\
 & = c \cdot [q\circ \delta] && \text{Lemma~\ref{l.sections}} \\
 & = c.
\end{align*}

Next, we compute, for $c \in \opk^\circ_T(X \rightarrow Z)$,
\begin{align*}
L(c) \cdot [g] & =  [\gamma] \cdot g^* c \cdot [g] && \text{(P1), associativity}\\
 & =  [\gamma] \cdot [p'] \cdot c && \text{(A2), flat pullbacks} \\
 & =  [p' \circ \gamma] \cdot c && \text{Lemma~\ref{l.sections}} \\
 & = c.
\end{align*}
\end{proof}

The next proposition says that Grothendieck groups of coherent sheaves appear as operational $K$-groups for projection to a point.

\begin{proposition} \label{p:smooth}
The map $\opk^\circ_T(X \rightarrow \pt) \rightarrow K_\circ^T(X)$ taking $c$ to $c_{\id}([\OO_\pt])$ is an isomorphism.
\end{proposition}

\begin{proof}
Similar to the proof of \cite[Proposition~17.3.1]{it}.  In the notation of that proof, the reduction to $\alpha = [V]$ is replaced by a reduction to $\alpha = [\OO_V]$, for $V$ and equivariant subvariety; this is justified by Lemma~\ref{l:devissage}.
\end{proof}

Multiplication defines a canonical map from $\opk_T^\circ(X) = \opk_T^\circ(X \xrightarrow{\id} X)$ to $\opk_T^\circ(X \to \pt)$, and composing with the  homomorphism of Proposition~\ref{p:smooth}, we get a canonical map $\opk_T^\circ(X) \to K^T_\circ(X)$.

\begin{corollary}  \label{c:smooth}
Suppose $X$ is smooth.  Then the canonical map
\[
\opk^\circ_T(X) \rightarrow K_\circ^T(X),
\]
taking $c$ to $c_{\id_X}([\OO_X])$, is an isomorphism.
\end{corollary}

\begin{proof}
Take $f$ to be $\id_X$ and $g\colon X \rightarrow \pt$ in Proposition~\ref{p:opk-smooth}.  Then apply Proposition~\ref{p:smooth} and the fact that $g^*([\OO_\pt]) = [\OO_X]$.
\end{proof}

\bigskip

Finally, we prove an $\AA^1$-homotopy invariance property for operational $K$-theory.

\begin{theorem}\label{t:homotopy}
Let $\pi: E \rightarrow Y$ be an equivariant affine bundle.  Then the pullback map
\[
  \opk_T^\circ(X \to Y) \to \opk_T^\circ(X\times_Y E \to E)
\]
is an isomorphism.  In particular, there is a natural isomorphism
\[
\opk_T^\circ(X \to Y) \to \opk_T^\circ(X \times \AA^1 \rightarrow Y \times \AA^1).
\]
\end{theorem}

\begin{proof}
The inverse map $\alpha\colon \opk_T^\circ(X\times_Y E \to E) \to  \opk_T^\circ(X \to Y)$ is defined as follows.  Given $g\colon Y' \to Y$, let $\tilde{g}\colon Y'\times_Y E \to E$ be the projection.  For $c\in \opk_T^\circ(X\times_Y E \to E)$, define $\alpha(c)$ by $\alpha(c)_g = c_{\tilde{g}}$, using the natural isomorphisms $K^T_\circ(Y') = K^T_\circ(Y'\times_Y E)$ and $K^T_\circ(X') = K^T_\circ(X'\times_Y E)$.
\end{proof}

\begin{corollary} \label{c:homotopy}
For any $X$, the pullback map gives a natural isomorphism $\opk^\circ_T(X) \cong \opk^\circ_T(X \times \AA^1)$.
\end{corollary}

\renewcommand{\labelenumi}{(\theenumi)}

\begin{remark}
Most of the constructions presented here work in equivariant $K$-theory for an arbitrary reductive group $G$.  However, the proof of Proposition~\ref{p:smooth} uses the fact that $K^T_\circ(X)$ is generated by classes of structure sheaves of invariant subschemes, which depends on $T$ being a torus.
\end{remark}

\begin{remark}
It would be interesting to construct a graded operational bivariant theory $\opk^*(X \to Y)$, using the (higher) algebraic $K$-theory of coherent sheaves $K_*(X)$ as the covariant component.  Operators in $\opk^i(X \to Y)$ should be homomorphisms $K_j(Y') \to K_{j+i}(X')$, for the usual fiber squares, subject to some compatibility with pullbacks by flat maps and regular embeddings.  A basic question is whether one recovers the operational $K$-theory we have defined above as the $0$th component of such a graded operational theory.
\end{remark}

\section{Computing operational $K$-theory}\label{s:kimura}

The remainder of this paper is concerned with the contravariant part of  bivariant operational $K$-theory, the $R(T)$-algebras 
\[
\opk^\circ_T(X) = \opk_T^\circ(X\xrightarrow{\id}X).
\]
These rings do not come with any natural presentation in terms of generators and relations.  Here, we develop tools for computing them inductively from the $K$-theory rings of smooth varieties using resolution of singularities, equivariant envelopes, and Corollary~\ref{c:smooth}.  The results in this section are closely analogous to those proved by Kimura for operational Chow theory in \cite{kimura}, and extended to the equivariant setting in \cite{eg-eit, chow}, but the proofs for operational $K$-theory are substantially different.  Notably, the proof of the exact sequence in Proposition~\ref{p:kim1} requires the ``dual" exact sequence for $K_\circ^T$ (Proposition~\ref{p:env-exact}), which is a descent theorem whose proof uses higher equivariant $K$-theory and simplicial schemes.

\begin{remark}
By taking care with completions with respect to the augmentation ideal of $R(T)$, one can deduce similar results with rational coefficients from Kimura's results, using the associated graded algebra of the natural filtration of $K_\circ^T(X)$ by dimension of support and the Riemann-Roch theorem for singular spaces.  Conversely, the rational coefficient versions of Kimura's results follow from ours, using the Riemann-Roch theorem.
\end{remark}

The following are analogues of Kimura's theorems for Chow cohomology \cite[Lemma~2.1, Theorem~2.3, and Theorem~3.1]{kimura}.  As in the Chow case, these statements are proved by using formal properties of bivariant theories together with Proposition~\ref{p:env-exact}.  In fact, granting Proposition~\ref{p:env-exact}, the proofs of Lemma~\ref{l:env-inj} and Propositions~\ref{p:kim1} and \ref{p:kim2} are analogous to the corresponding proofs in \cite{kimura}.

\begin{lemma}\label{l:env-inj}
If $f\colon X' \to X$ is an equivariant envelope, then $f^*\colon \opk_T^\circ(X) \to \opk_T^\circ(X')$ is injective.
\end{lemma}

\begin{proof}
Let $g\colon Y \to X$ be any equivariant map, and form the fiber square
\begin{diagram}
  Y' & \rTo^{f'} & Y \\
 \dTo^{g'} & &  \dTo_{g} \\
  X'  & \rTo^f    & X.
\end{diagram}
Then $f' \colon Y' \to Y$ is an equivariant envelope, so $f'_*$ is surjective by Proposition~\ref{p:env-exact}.  Injectivity of $f^*$ now follows from  Axiom (A1).
\end{proof}

\begin{proposition}\label{p:kim1}
Let $f\colon X' \to X$ be an equivariant envelope, and let $p_1,p_2$ be the projections $X' \times_X X' \to X'$.  The sequence
\[
  0 \to \opk_T^\circ(X) \xrightarrow{f^*} \opk_T^\circ(X') \xrightarrow{p_1^*-p_2^*} \opk_T^\circ(X' \times_X X')
\]
is exact.
\end{proposition}

\begin{proof}
We have already seen that $f^*$ is injective (Lemma~\ref{l:env-inj}), and $p_1^*f^*=p_2^*f^*$ by functoriality.  Given $c'\in \opk_T^\circ(X')$ such that $p_1^*c'=p_2^*c'$, it remains to find $c\in\opk_T^\circ(X)$ such that $c' = f^*c$.

Let $g\colon Y \to X$ be an equivariant map, and let $Y' = Y \times_X X'$, so $f'\colon Y' \to Y$ is also an envelope.  By Proposition~\ref{p:env-exact}, given $\alpha\in K^T_\circ(Y)$ we can find $\alpha'\in K^T_\circ(Y')$ with $\alpha=f'_*\alpha'$.  Now define $c$ by
\[
  c_g(\alpha) = f'_*(c'_{g'} (\alpha') ).
\]
To see that this is independent of the choice of $\alpha'$, it suffices to show that $f'_*(\beta') = 0$ implies $f'_*(c'_{g'}(\beta'))=0$.  By Proposition~\ref{p:env-exact}, find $\beta'' \in K^T_\circ(Y'\times_Y Y')$ such that $\beta' = p'_{1*}\beta'' - p'_{2*}\beta''$, where $p_i'\colon Y'\times_Y Y' \to Y$ are the projections.  Note that $p_1^*c'=p_2^*c'$ implies $(p'_1)^*(g')^*c' = (p'_1)^*(g')^*c'$.  Now compute:
\begin{align*}
  f'_*(c'_{g'}(\beta')) &= f'_*(c'_{g'}(p'_{1*}\beta'' - p'_{2*}\beta'')) \\
                        &= f'_*(p'_{1*}c'_{g'\circ p'_1}(\beta'') - p'_{2*}c'_{g'\circ p'_1}(\beta'')) \\
                        &= (f'\circ p_1)_*( (p'_1)^*(g')^*c'(\beta'') - (p'_2)^*(g')^*c'(\beta'') ) \\
                        &= 0.
\end{align*}

Finally, $c' = f^*c$, essentially by definition.  
\end{proof}

\begin{proposition}\label{p:kim2}
Suppose $f\colon X' \to X$ is a birational equivariant envelope, restricting to an isomorphism over an invariant open $U\subseteq X$.  Let $S_i$ be the irreducible components of $X\setminus U$, and write $f_i\colon E_i = f^{-1}S_i \to S_i$.  A class $c'\in \opk_T^\circ(X')$ lies in the image of $f^*$ if and only if the restriction $c'|_{E_i}$ lies in the image of $f_i^*$ for all $i$.
\end{proposition}

\begin{proof}
The ``only if'' direction is obvious.  For the other direction, assume $c'\in \opk_T^\circ(X')$ satisfies the restriction hypothesis, so $c'|_{E_i}$ lies in the image of $f_i^*$ for all $i$.  By Proposition~\ref{p:kim1}, it will suffice to show that $p_1^*c'=p_2^*c'$ in $\opk_T^\circ(X' \times_X X')$.

For any equivariant map $g\colon Y \to X'\times_X X'$ and any element $\xi \in K^T_\circ(Y)$, we must show
\begin{equation}\label{p1p2}
  (p_1^*c')_g(\xi) = (p_2^*c')_g(\xi).
\end{equation}
It suffices to do this for $\xi = [\OO_V]$, for $V\subseteq Y$ an invariant subvariety, because these classes generate $K^T_\circ(Y)$ as an $R(T)$-module.  In fact, to check \eqref{p1p2} for $\xi=[\OO_V]$, we may replace $Y$ with $V$, using Axiom (A1).

If the image of $V$ in $X$ is contained in some $S_i$, we have
\begin{align*}
  (p_1^*c')_g(\xi) &= ((p_1|_{E_i\times_X E_i})^*(c'|_{E_i}))_g(\xi) \\
                       &= ((p_1|_{E_i\times_X E_i})^*(f_i^*c_i)_g(\xi) \\
                       &= (f_i\circ(p_1|_{E_i\times_X E_i}))^*c_i)_g(\xi).
\end{align*}
Noting that $f_i\circ(p_1|_{E_i\times_X E_i}) = f_i\circ(p_2|_{E_i\times_X E_i})$, we obtain $(p_1^*c')_g(\xi)=(p_2^*c')_g(\xi)$.

If the image of $V$ in $X$ is not contained in any $S_i$, then $V$ factors through the diagonal in $X' \times_X X'$.  Now $(p_1^*c')_g(\xi) = c'_{p_1\circ g}(\xi) = c'_{p_2\circ g}(\xi) = (p_2^*c')_g(\xi)$.
\end{proof}

\begin{remark}
The proposition can be rephrased as saying the sequence
\begin{equation}\label{e:kim-no}
  0 \to \opk_T^\circ(X) \to \opk_T^\circ (X') \oplus \bigoplus_i\opk_T^\circ (S_i) \to \bigoplus_i \opk_T^\circ (E_i)
\end{equation}
is exact.  It follows that the sequence
\begin{equation}\label{e:kim3}
  0 \to \opk_T^\circ(X) \to \opk_T^\circ (X') \oplus \opk_T^\circ (S) \to  \opk_T^\circ (E)
\end{equation}
is also exact, where $S = X\setminus U$ and $E=f^{-1}S$.  (By Lemma~\ref{l:env-inj}, the maps $\opk_T^\circ S \to \bigoplus \opk_T^\circ S_i$ and $\opk_T^\circ E \to \bigoplus \opk_T^\circ E_i$ are injective.)  Since the fiber square
\begin{diagram}
  E & \rInto & X' \\
 \dTo & &  \dTo_{f} \\
  S  & \rInto  & X.
\end{diagram}
is an abstract blowup diagram, \eqref{e:kim3} says that $\opk_T^\circ$ is a sheaf in the {\em cdh} topology.  This formulation is useful for comparing with other theories satisfying {\em cdh}-descent, such as $KH$-theory.  (See \cite{friedlander} for an introduction to this topology.)  
In fact, in categories where smooth envelopes exist---for example, when the base field has characteristic zero---Theorem~\ref{t.kh} will imply that $\opk_T^\circ$ is the {\em cdh}-sheafification of $K_T^\circ$.
\end{remark}

\begin{corollary}\label{c:opKtrivial}
Suppose $T=T_1\times T_2$ acts on $X$ such that the action of $T_1$ is trivial.  Assume there exists an equivariant envelope $f\colon X' \to X$, where $X'$ is smooth and $f$ is an isomorphism over a dense open set in $X$.  Then there is a canonical isomorphism $\opk_T^\circ(X) = R(T_1)\otimes \opk_{T_2}^\circ(X)$.
\end{corollary}

The proof is similar to that of \cite[Theorem~2]{eg-eit}, which implies the corresponding result for equivariant Chow cohomology.

\begin{proof}
If $X$ is smooth, this follows from Lemma~\ref{l:Ktrivial} and Corollary~\ref{c:smooth}.  For the general case, we may assume $T_1$ acts trivially on $X'$, so $\opk_T^\circ(X') = R(T_1)\otimes \opk_{T_2}^\circ(X')$.  Indeed, given any $X' \to X$ as in the hypothesis, consider this a $T_2$-equivariant map by restricting the action, and define a new $T$-action on $X'$ by letting $T_1$ act trivially.

Now let $U\subseteq X$ be an invariant open such that $f^{-1}U \to U$ is an isomorphism, and let $S = X \setminus U$ be the complement.  By noetherian induction, we assume $\opk_T^\circ(S) = R(T_1)\otimes \opk_{T_2}^\circ(S)$ and $\opk_T^\circ(E) = R(T_1)\otimes \opk_{T_2}^\circ(E)$, where $E=f^{-1}S$.  Using Proposition~\ref{p:kim2}, the same conditions characterize the images of $\opk_T^\circ(X)$ and $R(T_1)\otimes \opk_{T_2}^\circ(X)$ in $\opk_T^\circ(X')$.  In other words, the kernels of the horizontal arrows in the diagram
\begin{diagram}
\opk_T^\circ (X') \oplus \opk_T^\circ (S) & \rTo & \opk_T^\circ (E) \\ 
 \|  & & \| \\
\left( R(T_1) \otimes \opk_{T_2}^\circ (X') \right) \oplus  \left( R(T_1)\otimes \opk_{T_2}^\circ (S) \right) & \rTo  &  R(T_1)\otimes \opk_{T_2}^\circ (E)
\end{diagram}
are isomorphic, and these kernels are $\opk_T^\circ(X)$ and $R(T_1)\otimes \opk_{T_2}^\circ(X)$, respectively.
\end{proof}

Now we discuss more directly the relationship between operational $K$-theory, the usual $K$-theory of perfect complexes, and Weibel's homotopy $K$-theory.

\begin{theorem}\label{t:smooth}
For any $X$, there is a canonical homomorphism $K_T^\circ(X) \to \opk_T^\circ(X)$ of $R(T)$-algebras, sending a class $\alpha$ to the operator $[\alpha]$ which acts by $[\alpha]_g = g^*\alpha \cdot \xi$, for any $g\colon Y \to X$ and $\xi\in K^T_\circ(Y)$.  Together with the canonical map $\opk_T^\circ(X) \to K^T_\circ(X)$ of Proposition~\ref{p:smooth}, this factors the canonical homomorphism $K_T^\circ(X) \to K^T_\circ(X)$.

If $X$ is smooth, the homomorphisms $K_T^\circ(X) \to \opk_T^\circ(X) \to K^T_\circ(X)$ are all isomorphisms of $R(T)$-modules.
\end{theorem}

\begin{proof}
To see that the homomorphism $K_T^\circ(X) \to \opk_T^\circ(X)$ is well defined, we must check that for any $\alpha\in K_T^\circ(X)$ the operator $[\alpha]$ satisfies Axioms (A1) and (A2); that is, it commutes with  pushforward for proper maps and pullback for flat maps and regular embeddings.  These all follow from the Lemmas~\ref{l.push-commute} and \ref{l.commutes}, with $X = Y$ and $f = \id$, since the identity is a closed embedding.

That the canonical map $K_T^\circ(X) \to K^T_\circ(X)$ factors as claimed is easily seen from the definitions.  When $X$ is smooth, the homomorphisms $K_T^\circ(X) \to K^T_\circ(X)$ and $\opk_T^\circ(X) \to K^T_\circ(X)$ are isomorphisms (using Corollary~\ref{c:smooth} for the latter), so it follows that $K_T^\circ(X) \to \opk_T^\circ(X)$ is an isomorphism, as well.
\end{proof}

\begin{theorem}\label{t.kh}
Assume the base field has characteristic zero.  Let $L_T^\circ$ be any contravariant functor from $T$-schemes to groups that admits a natural transformation $\eta$ to $K_T^\circ$ when restricted to smooth schemes.  Then $\eta$ extends uniquely to a natural transformation from $L_T^\circ$ to $\opk_T^\circ$.
\end{theorem}

Characteristic zero is used only to guarantee the existence of a suitable smooth envelope, and the proof goes through whenever such envelopes exist.  In particular, the theorem holds for toric varieties over an arbitrary base field.

\begin{proof}
When $\dim X$ is zero, then $X$ is smooth and hence the natural map exists by hypothesis.  Proceed by induction on dimension.  Use resolution of singularities to construct a birational equivariant envelope $X' \to X$, with $X'$ smooth, and with the exceptional locus a simple normal crossings divisor.  Let $E$ and $S$ be as in the sequence \eqref{e:kim3}.  In the diagram
\begin{diagram}
0 & \rTo & \opk_T^\circ(X) & \rTo & \opk_T^\circ (X') \oplus \opk_T^\circ (S) & \rTo &  \opk_T^\circ (E) \\ 
  &      &   \uDashto &       &      \uTo     &    &  \uTo  \\
  &    &   L_T^\circ(X)   &  \rTo & L_T^\circ (X') \oplus  L_T^\circ (S) & \rTo  &  L_T^\circ (E) ,
\end{diagram}
the top row is exact, the rightmost vertical arrow exists by the induction hypothesis, and the middle vertical arrow exists by induction (for the $S$ factor) and the smooth case (for the $X'$ factor).  Since the composition of the two horizontal arrows in the bottom row is zero, the image of $L_T^\circ(X)$ in $\opk_T^\circ (X') \oplus \opk_T^\circ (S)$ lies in the kernel of the map to $\opk_T^\circ (E)$, which is $\opk_T^\circ(X)$.  This procedure constructs a natural and functorial map $L_T^\circ(X) \rightarrow \opk_T^\circ(X)$, for arbitrary $X$, as required.
\end{proof}

\begin{corollary}\label{c.kh}
Assume the base field has characteristic zero.  Then there is a natural  map $\theta\colon KH^\circ(X) \rightarrow \opk^\circ(X)$.
\end{corollary}

As in Theorem~\ref{t.kh}, the characteristic zero hypothesis can be replaced by an assumption that birational smooth envelopes exist, so in particular the same statement holds for toric varieties over a base field of arbitrary characteristic.

\begin{proof}
For smooth schemes $X$, the natural map from $K^\circ(X)$ to $KH^\circ(X)$ is an isomorphism, and its inverse provides the natural transformation required to apply Theorem~\ref{t.kh}.
\end{proof}

\begin{remark}
It would be interesting to develop an equivariant version of homotopy $K$-theory, with formal properties similar to those of $KH^*$; to our knowledge, this has not been explored.  One application would be an equivariant version of Corollary~\ref{c.kh}.
\end{remark}

\begin{remark}\label{r.kan}
In more abstract language, Theorem~\ref{t.kh} may be phrased simply as follows: the functor $\opk_T^\circ$ from the category of (separated and finite-type over $\kk$) $T$-schemes to groups is naturally isomorphic to the Kan extension of the functor $K_T^\circ$ from smooth $T$-schemes to groups, along the inclusion of the full category of smooth $T$-schemes inside all (separated and finite-type) $T$-schemes.  This is a fundamental characterization of operational $K$-theory.

A similar statement holds for any operational theory whenever one has an analogue of Kimura's exact sequence (Proposition~\ref{p:kim2}) and resolution of singularities.  For example, (equivariant) operational Chow cohomology is also a Kan extension, at least in characteristic zero.
\end{remark}

We now give two examples showing that the natural map $KH^\circ(X) \rightarrow \opk^\circ(X)$ is not an isomorphism in general.  A third will be given in \S\ref{s:proof}.

\begin{example} \label{ex:nodal}
Let $X$ be a nodal cubic curve in the affine plane $\AA^2$.  Then $X$ is seminormal, so its Picard group is $\AA^1$-homotopy invariant \cite{traverso}.  By \cite[Theorem~3.3]{weibelKH}, it follows that
\[
KH^\circ (X) = \ZZ \oplus \kk^*.
\]
Now, consider the natural diagram
\begin{diagram}
 \opk^\circ(X) & \rTo & \opk^\circ(\widetilde X) \\
 \uTo &  & \uTo \\
 KH^\circ(X)  & \rTo   & KH^\circ(\widetilde X),
\end{diagram}
where $\widetilde X \cong \AA^1$ is the normalization of $X$.  Both $\opk^\circ(\widetilde X)$ and $KH^\circ(\widetilde X)$ are canonically identified with $\ZZ$, and the arrow between them is an isomorphism.  Furthermore, it follows from Proposition~\ref{p:kim1} that the top horizontal arrow is an isomorphism, and hence the left vertical arrow is the canonical projection from $\ZZ \oplus \kk^*$ to $\ZZ$.  In particular, when $\kk=\CC$, $KH^\circ(X) \rightarrow \opk^\circ(X)$ has uncountable kernel.
\end{example}

\begin{example} \label{ex:normal}
Let $X$ be a variety of dimension $d$ over $\CC$ that is smooth away from an isolated singularity.  Let $\pi: \widetilde X \rightarrow X$ be a log resolution, so the exceptional locus $E$ is a divisor with simple normal crossings.  Let $\Delta(E)$ be the dual complex of $E$, i.e. the $\Delta$-complex with one vertex for each irreducible component of $E$, one edge for each component of a pairwise intersection, and so on, as described in \cite[\S 2]{boundary}.   Recall that $KH$ can be nonzero in negative degrees.  Haesemeyer's computations for negative $KH$ \cite[Theorem~7.10]{haesemeyer} show that
\[
KH^{-d}(X) \cong H^{d-1}(\Delta(E)).
\]
By Bass's fundamental theorem for $KH$ \cite[Theorem~6.11]{weibelKH}, this gives $H^{d-1}(\Delta(E))$ as a direct summand of $KH^\circ(X \times \GG_m^d)$ that is contained in the kernel of the pullback map to $KH^\circ(\widetilde X \times \GG_m^d)$.  In the resulting natural diagram,
\begin{diagram}
 \opk^\circ(X \times \GG_m^d) & \rTo & \opk^\circ(\widetilde X \times \GG_m^d) \\
 \uTo &  & \uTo \\
 KH^\circ(X \times \GG_m^d)  & \rTo   & KH^\circ(\widetilde X \times \GG_m^d),
\end{diagram}
the top horizontal arrow is an injection, by Proposition~\ref{p:kim1}.  It follows that the summand $H^{d-1}(\Delta(E)) \subset  KH^\circ(X \times \GG_m^d)$ is contained in the kernel of the map to $\opk^\circ(X \times \GG_m^d)$.

This general construction produces many nontrivial examples.  For instance, $X$ could be a deformation of a cone over a degenerate elliptic curve, with $E$ a loop of $\PP^1$'s, and $\Delta(E)$ a circle.  Or $X$ could be a $3$-fold with $\Delta(E)$ being a copy of $S^2$ or $\RR \PP^2$, as in \cite[Example~8.1]{boundary}.
\end{example}

\begin{remark}
The map $KH^\circ(X) \to \opk^\circ(X)$ factors the natural map $K^\circ(X) \to \opk^\circ(X)$.  Properties of the other factor, $K^\circ(X) \to KH^\circ(X)$, therefore give information about the map from the $K$-theory of perfect complexes to operational $K$-theory.

For example, let $X$ be a toric variety with fan $\Delta$.  By \cite[Proposition~5.6]{chww}, the map $K^*(X) \to KH^*(X)$ is a split surjection.  The proof shows that the Mayer-Vietoris sequence computing $KH^*(X)$ can be described quite explicitly as follows.  Let $\sigma_1,\ldots,\sigma_r$ be the maximal cones of $\Delta$, write $\tau(I) = \sigma_{i_0} \cap \cdots \cap \sigma_{i_p}$ for each subset $I\subset \{1,\ldots,r\}$, and let $O_\tau$ be the minimal orbit in the invariant affine open set $U_\tau$.  Then there are spectral sequences
\begin{align} \label{e.mv-kh}
  E^1_{p,q} &= \bigoplus_{I=\{i_0,\ldots,i_p\}} K^q(O_{\tau(I)}) \Rightarrow KH^{q-p}(X)
\intertext{and}
  E^1_{p,q} &= \bigoplus_{I=\{i_0,\ldots,i_p\}} K^q(U_{\tau(I)}) \Rightarrow K^{q-p}(X),
\end{align}
together with natural split surjections $K^q(U_{\tau(I)}) \to K^q(O_{\tau(I)})$.  
Since each $O_{\tau(I)}$ is isomorphic to a split torus, say of dimension $d(I)$, the term $K^q(O_{\tau(I)})$ is isomorphic to $\bigoplus_{i=0}^{d(I)} K^{q-i}(k)^{\oplus \binom{d(I)}{i}}$, where $k$ is the base field.  The differential $d^1$ can be computed easily from the projections $T \to O_{\tau(I)}$; it is similar to the differential of the Mayer-Vietoris spectral sequence for singular cohomology, with respect to the same open cover.  (We thank Burt Totaro for suggesting this description of \cite[Proposition~5.6]{chww}.)

When $X$ has a sufficiently nice stratification by torus orbits, there are similar spectral sequences computing $K^*(X)$ and $KH^*(X)$.  For example, one can take $X$ to be any $T$-invariant closed subset of a toric variety $Y$.  If $Y$ is covered by invariant affines $V_\sigma$, then one has the Mayer-Vietoris sequence for the cover of $X$ by $U_\sigma = V_\sigma \cap X$.  The same reasoning shows that in this case $KH^*(X)$ is computed in terms of the $K$-theory of the ground field, and that there is a split surjection $K^*(X) \to KH^*(X)$.
\end{remark}

\begin{remark}\label{r.kh-surface}
The Mayer-Vietoris sequence \eqref{e.mv-kh} gives a straightforward computation of the homotopy $K$-theory for a complete toric surface over an arbitrary field $k$, since it collapses at the $E^2$ page.  Let $X$ be such a surface, corresponding to a complete fan with rays spanned by primitive integer vectors $(a_1,b_1),\ldots,(a_r,b_r)$.  Let $\mu$ be the index of the sublattice of $\ZZ^2$ spanned by these vectors.  Then one computes
\[
  KH^i(X) = K^i(k)^{\oplus r} \oplus K^{i+1}(k)/\mu\cdot K^{i+1}(k),
\]
where ``$\mu\cdot$'' denotes the operation of scaling an abelian group by $\mu$.  In particular, $KH^\circ(X) = \ZZ^{\oplus r} \oplus k^*/(k^*)^\mu$, and $KH^{-1}(X) = \ZZ/\mu\ZZ$.  (For algebraically closed fields $k$, we have $k^*=(k^*)^\mu$, so in this case $KH^\circ(X)$ depends only on the number of rays.)
\end{remark}

\section{Kronecker duality for $T$-linear varieties}

Now we state and prove an equivariant version of Theorem~\ref{t:simple-duality}.   Let the class of \define{$T$-linear varieties} be the smallest class of varieties such that any affine space with a $T$-action is a $T$-linear variety, the complement of a $T$-linear variety equivariantly embedded in affine space is $T$-linear, and any variety stratified into a finite disjoint union of $T$-linear varieties is $T$-linear.  The notion of stratification we use here is as follows: stratifying $X$ by locally closed subvarieties $S_i^\circ$ means $X$ is written as $X = \coprod S_i^\circ$, with $S_i = \overline{S_i^\circ}$ denoting the closure of a stratum, such that the complement $S_i \setminus S_i^\circ$ is contained in the union of the strata of dimension less than $\dim S_i$ (see \cite{totaro}).  A $T$-linear variety is also $T'$-linear for any subtorus $T'\subseteq T$.

\begin{theorem} \label{t:duality}
For a complete $T$-linear variety $X$, the natural map from $\opk^\circ_T(X)$ to $\Hom_{R(T)}(K^T_\circ(X), R(T))$, induced by pushforward to a point, is an isomorphism.
\end{theorem}

The proof has the same structure as that of the corresponding result for Chow cohomology (\cite{fmss}, \cite{totaro}), the main difference being the use of $K$-theory spectra, rather than higher Chow groups, to establish the K\"unneth isomorphism in Proposition~\ref{p.kunneth}, below.

Before proving the theorem, we observe that finite generation of the operational $K$ groups is a consequence, by the following lemma.

\begin{lemma}\label{l.linear-fg}
If $X$ is a $T$-linear variety, then $K^T_\circ(X)$ is a finitely-generated $R(T)$-module.
\end{lemma}

\begin{proof}
We use induction on dimension and the number of irreducible components.  First suppose $X$ is irreducible of dimension $d$, and write
\[
  X = S^\circ \cup \coprod_{\dim S_i <d} S^\circ_i
\]
Observe that the top stratum $S^\circ$ in an irreducible $T$-linear variety must be the complement of a (lower-dimensional) $T$-linear variety in an affine space, by the inductive definition.  Such varieties clearly have finitely generated $K^T_\circ$.  Also, $Z=\coprod_{\dim S_i <d} S^\circ_i$ is a closed $T$-linear subvariety of smaller dimension, so $K^T_\circ(Z)$ is finitely generated.  By the exact sequence
\[
  K^T_\circ(Z) \to K^T_\circ(X) \to K^T_\circ(S^\circ) \to 0,
\]
it follows that $K^T_\circ(X)$ is finitely generated.

If $X$ is reducible, one can find a top-dimensional stratum $S^\circ$ such that $X' = X\setminus S^\circ$ is a closed $T$-linear variety with fewer components, and apply the exact sequence again.
\end{proof}

As in Totaro's analogous result for Chow cohomology \cite{totaro}, Theorem~\ref{t:duality} is an immediate consequence of two facts.  The first is an analogue of \cite[Proposition~3]{fmss}.

\begin{proposition}\label{p.fmss}
Suppose $X$ is a complete $T$-variety with the property that for all $T$-varieties $Y$, the map
\[
  K^T_\circ(X) \otimes_{R(T)} K^T_\circ(Y) \to K^T_\circ(X\times Y)
\]
is an isomorphism.  Then the duality map
\[
\opk_T^\circ(X) \to \Hom_{R(T)}(K^T_\circ(X), R(T))
\]
is an isomorphism.
\end{proposition}

As in \cite{fmss}, the proof is formal.

\begin{proof}[Sketch of proof]
To define the inverse to the duality map, given a homomorphism $\phi\colon K^T_\circ(X) \to R(T)$, we construct an element $c_\phi$ of $\opk_T^\circ(X)$ as follows.  For a map $g\colon Y \to X$, the corresponding map $(c_\phi)_g\colon K^T_\circ(Y) \to K^T_\circ(Y)$ is the composition of the proper pushforward along the graph of $g$,
\[
  K^T_\circ(Y) \to K^T_\circ(X\times Y) = K^T_\circ(X)\otimes_{R(T)} K^T_\circ(Y)
\]
followed by
\[
K^T_\circ(X)\otimes_{R(T)} K^T_\circ(Y) \xrightarrow{\phi\otimes 1} R(T) \otimes_{R(T)} K^T_\circ(Y) = K^T_\circ(Y).
\]
The verification that $c_\phi$ satisfies the compatibility axioms is the same as in \cite{fmss}, as is the proof that this map is inverse to the duality map.
\end{proof}

The second fact we need is an analogue of \cite[Proposition~1]{totaro}.

\begin{proposition}\label{p.kunneth}
Let $X$ be a $T$-linear variety, and let $Y$ be an arbitrary $T$-variety.  Then the map
\[
  K^T_\circ(X) \otimes_{R(T)} K^T_\circ(Y) \to K^T_\circ(X\times Y)
\]
is an isomophism.
\end{proposition}

\begin{proof}
The proof follows Totaro's inductive argument, with Chow groups replaced by equivariant $K$-theory and the groups $CH^{*,*}(-,-,1)$ replaced by $\tilde K_1^T (-, -)$, which are defined as follows.  Let $X$ and $Y$ be arbitrary $T$-varieties.  Let $\KK'(T,X)$ and $\KK'(T,Y)$ be the equivariant $K$-theory spectra corresponding to categories of $T$-equivariant coherent sheaves (cf.~\cite{thomason}), and write $\rR=\KK'(T,\mathrm{pt})$ for the equivariant $K$-theory spectrum of the point $\Spec k$.  Then $\rR$ is a ring spectrum, and the other spectra are modules over $\rR$.  Define
\[
  \tilde\KK^T(X,Y) = \KK'(T,X) \wedge_\rR \KK'(T,Y),
\]
and set $\tilde{K}^T_q(X,Y) = \pi_q( \tilde\KK^T(X,Y) )$.  By \cite[Theorem~IV.6.4]{ekmm}, we have $\tilde{K}^T_\circ(X,Y) = K^T_\circ(X)\otimes_{R(T)} K^T_\circ(Y)$.  

Let $Z \subseteq X$ be a closed subvariety and let $U = X\setminus Z$.  We replace the two four column diagrams of Chow groups on \cite[p.~11]{totaro} with
\begin{diagram}
K^T_\circ(Z) \mathop\otimes_{R(T)} K^T_\circ(Y) & \rTo & K^T_\circ(X)\mathop\otimes_{R(T)} K^T_\circ(Y) & \rTo &  K^T_\circ(U)\mathop\otimes_{R(T)} K^T_\circ(Y) & \rTo & 0 \\
\dTo &  & \dTo &  & \dTo &  &  \\
K^T_\circ(Z\times Y) & \rTo & K^T_\circ(X \times Y) & \rTo &  K^T_\circ(U\times Y) & \rTo & 0,
\end{diagram}
and its continuation to the left, which is

{\footnotesize
\begin{diagram}
\tilde{K}^T_1(Z,Y) & \rTo & \tilde{K}^T_1(X,Y) & \rTo & \tilde{K}^T_1(U,Y) & \rTo & K^T_\circ(Z) \mathop\otimes_{R(T)} K^T_\circ(Y) & \rTo & K^T_\circ(X)\mathop\otimes_{R(T)} K^T_\circ(Y) \\ 
\dTo & &  \dTo &  & \dTo &  & \dTo &  & \dTo &  & \\
K_1^T(Z \times Y) & \rTo & K^T_1(X\times Y) & \rTo & K^T_1(U\times Y) & \rTo & K^T_\circ(Z\times Y) & \rTo & K^T_\circ(X \times Y) . 
\end{diagram}
}

\noindent The bottom rows are just the long exact localization sequence.  The smash product preserves fibrations, so we have a fibration 
\[
\tilde\KK^T(Z,Y) \to \tilde\KK^T(X,Y) \to \tilde\KK^T(U,Y),
\]
and the top rows of the above diagrams are the corresponding long exact homotopy sequence.  The outer tensor product defines a canonical map $\tilde\KK^T(X,Y) \to \KK'(T,X\times Y)$, inducing the vertical arrows.

The remainder of the proof follows Totaro's argument essentially verbatim.  We consider two properties for a $T$-variety $X$, namely
\begin{enumerate}
\item The natural map $K^T_\circ(X) \otimes_{R(T)} K^T_\circ(Y) \to K^T_\circ(X\times Y)$ is an isomorphism for all $T$-varieties $Y$, and
\item The natural map $\tilde{K}_1^T(X,Y) \rightarrow K_1^T(X \times Y)$ is surjective for all $T$-varieties $Y$.
\end{enumerate}
An induction on dimension shows simultaneously that both (1) and (2) hold for the complement of any $T$-linear variety embedded in affine space, and that (1) holds for arbitrary $T$-linear varieties, the latter being the statement that is to be proved.  The base case of the induction is that both properties hold for affine space.  Then one simple diagram chase shows that if (1) and (2) hold for $X$ and (1) holds for $Z$ then both (1) and (2) hold for $U$, and another diagram chase shows that if (1) and (2) hold for $U$ and (1) holds for $Z$, then (1) holds for $X$.
\end{proof}

\section{Operational $K$-theory of toric varieties}\label{s:proof}

We now prove Theorem~\ref{t:plp} by resolution of singularities and induction on dimension.  The argument is similar to the proof of \cite[Theorem~1]{chow}, using the $K$-theoretic results from Section~\ref{s:kimura} in place of Kimura's analogous results for Chow cohomology.

\begin{proof}[Proof of Theorem~\ref{t:plp}.]
As mentioned in the introduction, the theorem is known if $X$ is smooth.  If $X$ is singular, then there is a sequence
\[
  X_r \rightarrow X_{r-1} \rightarrow \cdots \rightarrow X_1  \stackrel{\pi}{\rightarrow} X_0 = X
\]
where $X_r$ is smooth, each $X_i$ is a toric variety, and the map $X_{i + 1} \rightarrow X_i$ is the blowup along a smooth $T$-invariant subvariety of $X_i$.  We proceed by induction on $r$ and the dimension of $X$.

Let $X' = X(\Delta') = X_1$, and let $\iota_{\sigma'}$ denote the inclusion of the torus orbit $O_{\sigma'}$ in $X'$.  By induction on $r$, we may assume that $\bigoplus_{\sigma' \in \Delta'} \iota_{\sigma'}^*$ maps $\opk_T^\circ(X')$ isomorphically onto $\PExp(\Delta')$.  Since $\pi$ maps $O_{\sigma'}$ isomorphically onto $O(\sigma)$ if $\sigma'$ is a maximal cone in the subdivision of $\sigma$ induced by $\Delta'$, it follows that $\bigoplus_{\sigma \in \Delta} \iota_\sigma^*$ maps $\opk_T^\circ(X)$ injectively into $\PExp(\Delta)$.  It remains to show that every integral piecewise exponential function on $\Delta$ is in the image of $\opk_T^\circ(X)$.

Say $\pi$ is the blowup along $V(\tau) \subset X$, and $V(\rho) = \pi^{-1} (V(\tau))$.  Let $\Star \tau$ be the set of cones in $\Delta$ containing $\tau$, and let $\Delta_\tau$ be the fan whose cones are the projections of cones in $\Star \tau$ to $(N / N_\tau)_\RR$, where $N_\tau$ is the sublattice generated by $\tau \cap N$.  Then $V(\tau)$ is the toric variety associated to $\Delta_\tau$ \cite[Section 3.1]{Fulton}.  By induction on dimension, we may assume $\opk_{T_\tau}^\circ(V(\tau)) \cong \PExp(\Delta_\tau)$, where $T_\tau$ is the dense torus in $V(\tau)$.  Choosing a splitting $T \cong T'\times T_\tau$, we have
\[
\opk_T^\circ(V(\tau)) \ \cong \ R(T') \otimes \opk_{T_\tau}(V(\tau)) \ \cong \ \PExp(\Star \tau),
\]
using Corollary~\ref{c:opKtrivial} for the first isomorphism.  
Similarly, the operational equivariant $K$ ring of the exceptional divisor is $\opk_T^\circ(V(\rho)) \cong \PExp(\Star \rho)$.

Note that $\Star \rho$ is a subdivision of $\Star \tau$, and $\Delta$ and $\Delta'$ coincide away from $\Star \tau$ and $\Star \rho$, and the blowup $\pi$ is an envelope \cite[Lemma~1]{chow}.  Then, by Proposition~\ref{p:kim2}, a class in $\opk_T^\circ(X')$ is in the image of $\opk_T^\circ(X)$ if and only if  its restriction to $V(\rho)$ is in the image of $\opk_T^\circ(V(\tau))$.  Therefore, a piecewise exponential function on $\Delta'$ is in the image of $\opk_T^\circ(X)$ if and only if its restriction to $\Star \rho$ is the pullback of a piecewise exponential function on $\Star \tau$.  In particular, the pullback of any piecewise exponential on $\Delta$ is in the image of $\opk_T^\circ(X)$, as required.
\end{proof}

Finally, we prove Theorem~\ref{t.toric-nontrivial}.  As remarked in the introduction, it suffices to establish the following:

\begin{theorem}\label{t.toric3fold}
For a three-dimensional toric variety $X$ over an algebraically closed field, the map $KH^\circ(X) \to \opk^\circ(X)$ is surjective.
\end{theorem}

The proof requires some facts about homotopy $K$-theory.  By \cite[Theorem~3.5]{haesemeyer}, there is a {\em cdh}-descent sequence
\begin{equation}\label{e.cdh}
 \to KH^{i+1}(E) \to KH^i(X) \to KH^i(X') \oplus KH^i(S) \to KH^i(E) \to 
\end{equation}
associated to any abstract blowup square
\begin{diagram}
  E & \rInto & X' \\
  \dTo &   &  \dTo_\pi \\
  S & \rInto & X.
\end{diagram}
That is, $\pi\colon X' \to X$ is a proper map, with $E = \pi^{-1}S$, inducing an isomorphism $X'\setminus E \to X \setminus S$.  In particular, we have a sequence \eqref{e.cdh} associated to any toric resolution of singularities.

To compare homotopy $K$-theory and operational $K$-theory, our main tool is the following commutative diagram with exact rows coming from the Kimura exact sequence \eqref{e:kim3} and the {\em cdh}-descent sequence \eqref{e.cdh}:
\begin{equation}\label{e.descent}
\begin{diagram}
      0   & \rTo & \opk^\circ(X) &\rTo & \opk^\circ(X') \oplus \opk^\circ(S) & \rTo  & \opk^\circ(E)  \\
      \uTo   &       &  \uTo  &     &  \uTo                  &        &  \uTo    \\ 
  KH^1(E) & \rTo & KH^\circ(X) & \rTo & KH^\circ(X') \oplus KH^\circ(S) & \rTo  & KH^\circ(E)  \ .
\end{diagram}
\end{equation}
The vertical arrows in this diagram come from the natural transformation $\theta: KH^\circ \rightarrow \opk^\circ$ given by Corollary~\ref{c.kh}.

We apply the diagram to study the natural map $\theta: KH^\circ(X) \rightarrow \opk^\circ(X)$ first for a chain of rational curves, then for a toric surface, and finally for a toric threefold.

A {\em chain of $n$ rational curves} is the reducible nodal variety
\[
X_1 \cup \cdots \cup X_n,
\]
whose irreducible components $X_i$ are smooth rational curves, constructed inductively by gluing $X_n$ to a chain of $n-1$ rational curves at a smooth point in $X_{n-1}$.

\begin{lemma}\label{l.curve-kh-opk}
Let $X$ be a chain of $n$ rational curves.  Then the natural map $\theta\colon KH^\circ(X) \to \opk^\circ(X)$ is an isomorphism.
\end{lemma}

\begin{proof}
When $n=1$, the isomorphism is clear since $X$ is smooth. We proceed by induction on $n$.  Let $S$ be the point of intersection where $X_n$ meets $X_{n-1}$, and let $X' = (X_1 \cup \cdots \cup X_{n-1}) \sqcup X_n$ be the disconnected curve obtained by pulling apart this node.     By induction we may assume that $KH^\circ(X') \rightarrow \opk^\circ(X')$ is an isomorphism.  Let $E = S \sqcup S$ be the preimage of $S$ under the gluing map $X' \rightarrow X$.

The map $KH^i(X') \oplus KH^i(S) \to KH^i(E)$ is surjective for all $i$, and it follows from the {\em cdh}-descent sequence \eqref{e.cdh} that the map $KH^\circ(X) \to KH^\circ(X') \oplus KH^\circ(S)$ is injective.  
Therefore, we can replace $KH^1(E)$ by $0$ in \eqref{e.descent}, and still have a commutative diagram with exact rows.  The first, third, and fourth vertical arrows are isomorphisms, so the second vertical arrow is also an isomorphism, by the five lemma.
\end{proof}

Slightly more generally, if $X$ is a disconnected union of chains of rational curves, it is a seminormal $1$-dimensional scheme, so $KH^\circ(X)$ is identified with $\Pic(X) \oplus H^0(X)$, where $H^0$ denotes the free abelian group on connected components \cite[IV.12.5.2]{weibel}.  A direct calculation using the Mayer-Vietoris sequence \eqref{e.mv-kh} shows that 
\begin{equation}\label{e.chain}
KH^i(X) = \left(\Pic(X) \otimes K^i(k)\right) \oplus \left(H^0(X) \otimes K^i(k)\right)
\end{equation}
for all $i$.  The isomorphism class of a line bundle on $X$ is determined by the degree of its restriction to each irreducible component, so $\Pic(X) \isom \ZZ^{\oplus n}$, where $n$ is the total number of irreducible components.

\begin{lemma}\label{l.surface-kh}
Let $X$ be a toric surface over an algebraically closed field, let $S$ be the singular set of $X$, and let $X' \to X$ be a toric resolution of singularities which is an isomorphism away from $S$.  Then the map $KH^\circ(X) \to KH^\circ(X') \oplus KH^\circ(S)$ is injective.
\end{lemma}

\begin{proof}
Let $E\subset X'$ be the exceptional divisor.  By the {\em cdh}-descent sequence \eqref{e.cdh}, it suffices to prove $KH^1(X') \oplus KH^1(S) \to KH^1(E)$ is surjective.  We have $KH^1(E) = (\Pic(E) \oplus H^0(E))\otimes k^*$, since $E$ is a union of chains of rational curves.  Since $S$ consists of finitely many points, one corresponding to each connected component of $E$, the map $KH^1(S) \to KH^1(E)$ is an isomorphism onto $H^0(E)\otimes k^*$.  The map $\Pic(X')\otimes k^* \to \Pic(E)\otimes k^*$ factors through the map $KH^1(X') \to KH^1(E)$.  Therefore, since $k^*$ is divisible, to prove that $KH^1(X') \oplus KH^1(S) \rightarrow KH^1(E)$ is surjective, it suffices to prove that the map $\Pic(X')\otimes \QQ \to \Pic(E)\otimes \QQ$ is surjective.  
The intersection matrix of $E$ in the surface $X'$ is negative-definite, so the subspace of $\Pic(X')\otimes\QQ$ spanned by the irreducible components of $E$ surjects onto $\Pic(E)\otimes\QQ$, and the lemma follows.
\end{proof}

\begin{proposition}\label{p.surface-kh-opk}
Let $X$ be a toric surface over an algebraically closed field.  Then the natural map $\theta: KH^\circ(X) \to \opk^\circ(X)$ is an isomorphism.
\end{proposition}

\begin{proof}
Let $S \subset X$ be the singular set, and let $X' \to X$ be a toric resolution of singularities which is an isomorphism away from $S$, with exceptional divisor $E$.  Then each connected component of $E$ is a chain of rational curves, so $KH^\circ(E) = \opk^\circ(E)$, by Lemma~\ref{l.curve-kh-opk}..  Furthermore, $KH^\circ(X') = \opk^\circ(X')$ and $KH^\circ(S)=\opk^\circ(S)$, since these are smooth varieties.  From the diagram \eqref{e.descent}, it follows that $KH^\circ(X) \to \opk^\circ(X)$ is surjective.  But applying Lemma~\ref{l.surface-kh}, we see this map is also injective.
\end{proof}

We can now complete the proof of Theorem~\ref{t.toric3fold} for a toric threefold $X$ over an algebraically closed field, using induction on the number of blowups along smooth $T$-invariant centers required to resolve the singularities of $X$.

\begin{proof}[Proof of Theorem~\ref{t.toric3fold}]
Let $X' \to X$ be a blowup along a smooth $T$-invariant center $S$, with exceptional divisor $E$, forming the first step in a resolution of singularities as in the proof of Theorem~\ref{t:plp}.  Consider the corresponding diagram \eqref{e.descent}.  By induction, $\theta: KH^\circ(X') \rightarrow \opk^\circ(X')$ is surjective, and $KH^\circ(S) = \opk^\circ(S)$ since $S$ is smooth, so the middle vertical arrow is surjective.  The rightmost vertical arrow $KH^\circ(E) \to \opk^\circ(E)$ is an isomorphism by Proposition~\ref{p.surface-kh-opk}, since $E$ is a toric surface.  The theorem now follows from the five-lemma.
\end{proof}

It is easy to construct examples of projective toric threefolds $X$ such that $KH^\circ(X)$ contains a factor of $k^*$.  For such varieties, the map $KH^\circ(X) \to \opk^\circ(X)$ is not injective.

\begin{example}\label{ex:toric3}
Let $\Delta$ be the fan over the faces of the triangular prism with vertices $(1,0,0)$, $(0,1,0)$, $(0,0,1)$, $(0,-1,-1)$, $(-1,0,-1)$, and $(-1,-1,0)$.  The singular locus $S$ of toric variety $X=X(\Delta)$ consists of four fixed points, corresponding to the maximal cones over the three rectangular faces and the simplical cone spanned by $(0,-1,-1)$, $(-1,0,-1)$, and $(-1,-1,0)$.  Let $\Delta'$ be the fan obtained from $\Delta$ by triangulating the three rectangular faces (via any of the $8$ possible choices) and subdividing one the other singular cone by adding a ray through $(-1,-1,-1)$.  The resulting toric variety $X'=X(\Delta')$ is smooth, and the exceptional locus $E$ of the map $X' \to X$ consists of four disjoint components, three isomorphic to $\PP^1$ and one isomorphic to $\PP^2$.  The map $\Pic(X') \to \Pic(E)$ is given by a $4\times 4$ matrix of rank $3$, so the corresponding map $KH^1(X') \oplus KH^1(S) \to KH^1(E)$ has a copy of $k^*$ in its cokernel, mapping nontrivially to $KH^\circ(X)$ in the {\em cdh}-descent sequence.
\end{example}

\appendix


\section{Descent for equivariant $K$-theory}\label{s:descent}

All schemes in this appendix are separated and of finite type over a fixed field.  Our goal is to establish an equivariant version of a theorem of Gillet \cite[Corollary~4.4]{gillet}.

\begin{theorem}\label{t.gillet}
Let $G$ be a connected algebraic group.  
Let $X \to Y$ be a proper morphism of $G$-schemes, and let $Z \to X$ be an equivariant Chow envelope (relative to $Y$), that is, an equivariant envelope such that $Z$ is projective over $Y$.  Then the sequence
\[
  K^G_\circ(Z\times_X Z) \to K^G_\circ(Z) \to K^G_\circ(X) \to 0
\]
is exact, where the first map is the difference of the pushforwards by the two projections.
\end{theorem}

The proof is almost entirely a repetition of Gillet's arguments.  For the convenience of the reader, we outline it here, with detailed references for the main points.

We first set up some terminology concerning simplicial schemes, and refer to \cite{bk} or \cite{conrad} for the basic facts.  A simplicial scheme is a contravariant functor from the category $\scat$ (of finite ordinals with non-decreasing maps) to the category of schemes.  The category of simplicial schemes is equipped with skeleton and coskeleton functors; we will not define these here, except to say that the $0$th coskeleton of an augmented simplicial scheme $Z_\bullet \to X$ is the simplicial scheme $\cosk^X_0(Z_\bullet)$ with $n$th term given by the $(n+1)$-fold fiber product of $Z$ over $X$.  See \cite{conrad} for the general definitions.

A group $G$ acts on a simplicial scheme $Z_\bullet$ in the evident way, by acting on all $Z_n$ equivariantly for the structure maps.  An \define{equivariant hyperenvelope} is an equivariant augmented simplicial scheme $Z_\bullet \to X$ such that $Z_0 \to X$ is an equivariant envelope, and for all $i\geq 0$, the map
\[
  Z_i \to \cosk^X_i(\sk_{i-1}Z_\bullet)
\]
is an equivariant envelope.  A \define{projective equivariant hyperenvelope} is an equivariant hyperenvelope for which all maps $Z_n \to X$ are projective.  

To define the equivariant $K$-theory of a simplicial scheme, following Gillet, consider the following condition on sheaves of $G$-equivariant $\OO_{Z_n}$-modules:
\renewcommand{\theenumi}{\fnsymbol{enumi}}
\begin{enumerate}
\item For all $\tau\colon \bm \to \bn$ in $\scat$ and all $p>0$, we have $R^p\tau_*\shfF=0$ as sheaves on $Z_m$. \label{g-cond}
\end{enumerate}
\renewcommand{\theenumi}{\arabic{enumi}}
Let $\catA^G(Z_n)$ be the full exact subcategory of $(\catCoh^G_{Z_n})$ formed by the sheaves satisfying \eqref{g-cond}.  Putting these together, we get a simplicial category $\catA^G(Z_\bullet)$.  Condition \eqref{g-cond} makes each $\tau_*\colon \catA^G(Z_n) \to \catA^G(Z_m)$ an exact functor, so we obtain another simplicial category $Q\catA^G(Z_\bullet)$ by applying Quillen's construction \cite{quillen}.  The \define{equivariant $K$-theory spectrum} of $Z_\bullet$ is the simplicial spectrum
\[
  \KK'(G,Z_\bullet) = \Omega|NQ\catA^G(Z_\bullet)|,
\]
where $N$ and $|\cdot|$ denote the nerve and geometric realization functors, respectively.  We define $K^G_q(Z_\bullet) = \pi_q(\KK'(G,Z_\bullet))$.

There is a general spectral sequence for computing with simplicial schemes; see \cite[3.14]{thomason-ss}, \cite[XII.5.7]{bk}, or \cite[15.10]{dugger}.  In our context, it takes the following form:

\begin{lemma}\label{l.ss-eqk}
Let $Z_\bullet$ be an equivariant simplicial scheme, quasi-projective over a base $Y$, with projective face maps.  Then there is a convergent spectral sequence
\[
  E^1_{p,q} = K^G_q(Z_p) \Rightarrow K^G_{p+q}(Z_\bullet).
\]
The differential $d^1\colon K^G_q(Z_p) \to K^G_q(Z_{p-1})$ is the alternating sum of the face maps.
\end{lemma}

We also need an equivariant version of one of the main theorems of \cite{gillet}.

\begin{theorem}\label{t.gillet-main}
Let $p\colon Z_\bullet \to X$ be a projective equivariant hyperenvelope.  Then $p_*\colon K^G_q(Z_\bullet) \to K^G_q(X)$ is an isomorphism for all $q$.
\end{theorem}

\begin{proof}[Sketch of proof]
The proof follows Gillet's argument, proceeding in three steps.

{\bf First step:} Given a projective equivariant envelope $Z\to X$, take $Z_\bullet = \cosk^X_0(Z)$, so $Z_n$ is the $(n+1)$-fold fiber product of $Z$ over $X$.  Also assume $X=G/H$ is a homogeneous space defined over a field $F$, so $H\subseteq G$ is a closed subgroup defined over $F$.

Since $p\colon Z \to X=G/H$ is an equivariant envelope, there is an invariant subvariety $\tilde{X}\subseteq Z$ mapping birationally and equivariantly to $X$.  Such a map $\tilde{X} \to X$ is an isomorphism, so there is an equivariant section $X \to Z$.  This extends to an equivariant section $s\colon X \to Z_\bullet$ (regarding $X$ as a constant simplicial scheme).  Since $Z_\bullet$ is a $0$-coskeleton, the maps $s$ and $p$ are homotopy-inverses (\cite[Lemma~5.7]{conrad}), so the induced map of simplicial groups $K^G_q(Z_n) \to K^G_q(X)$ is a homotopy equivalence.  Now it follows from the spectral sequence of Lemma~\ref{l.ss-eqk} that $K^G_q(Z_\bullet) \to K^G_q(X)$ is an isomorphism.

{\bf Second step:} Continue to assume $Z_\bullet = \cosk^X_0(Z)$ is a $0$-coskeleton, but now allow $X$ to be arbitrary.

The argument for this step proceeds exactly as in \cite{gillet}, except that the noetherian induction is taken over $G$-invariant closed subschemes $Y\subseteq X$.  The base case is when $Y=G/H$ is an orbit, which is taken care of by the first step.

{\bf Third step:} The general case stated in the theorem.

Here we follow \cite{gillet} verbatim: Let $Z_\bullet[i] = \cosk^X_i(\sk_i Z_\bullet)$, and use induction on $i$, starting from the base case $i=0$, which is the situation of the second step.
\end{proof}

Theorem~\ref{t.gillet} now follows easily:

\begin{proof}[Proof of Theorem~\ref{t.gillet}]
Look at the last terms in the spectral sequence of Lemma~\ref{l.ss-eqk}.  We have
\[
  E^2_{0,0} = K^G_\circ(Z_0)/\mathrm{im}(K^G_\circ(Z_1) \to K^G_\circ(Z_0)).
\]
Apply this to the hyperenvelope $Z_\bullet \to X$, with $Z_\bullet = \cosk^X_0(Z)$, so $Z_0=Z$ and $Z_1 = Z \times_X Z$.  Together with the isomorphism of Theorem~\ref{t.gillet-main}, this gives the exact sequence of Theorem~\ref{t.gillet}.  (In fact, we only needed the first two steps in the proof of Theorem~\ref{t.gillet-main}.)
\end{proof}




\begin{thebibliography}{CHWW}


\bibitem[A1]{al1} A.~Al Amrani, ``Groupe de Chow et K-th\'eorie coh\'erente des espaces projectifs tordus,'' Proceedings of Research Symposium on K-Theory and its Applications (Ibadan, 1987), {\em K-Theory} {\bf 2} (1989), no. 5, 579--602.

\bibitem[A2]{al2} A.~Al Amrani, ``Complex K-theory of weighted projective spaces,'' {\em J. Pure Appl. Algebra} {\bf 93} (1994), no. 2, 113--127.

\bibitem[AS]{atiyah-segal} M.~Atiyah and G.~Segal, ``The index of elliptic operators. II.'' {\em Ann. of Math. (2)} {\bf 87} (1968), 531--545.

\bibitem[SGA6]{sga6} P.~Berthelot, A.~Grothendieck, and L.~Illusie, {\em Seminaire de G\'eom\'etrie Alg\'ebrique 6: Theorie des intersections et theoreme de Riemann-Roch}, 1966-1967.

\bibitem[Bo]{bourbaki} N.~Bourbaki, {\em Alg\`ebre commutative}, \'Elements de Math\'emathique, Chap. 1--7, Hermann, Paris, 1961--1965.

\bibitem[BK]{bk} A.~Bousfield and D.~Kan, {\em Homotopy limits, completions and localizations}, Lecture Notes in Math., Vol.~304, Springer-Verlag, Berlin-New York, 1972.

\bibitem[BV]{bv} M.~Brion and M.~Vergne, ``An equivariant Riemann-Roch theorem for complete, simplicial toric varieties,'' {\em J.~Reine Angew.~Math.} {\bf 482} (1997), 67--92.

\bibitem[CG]{cg} N.~Chriss and V.~Ginzburg, \emph{Representation Theory and Complex Geometry}, Birkh\"auser, 1997.

\bibitem[Co]{conrad} B.~Conrad, ``Cohomological descent,'' notes available online at \texttt{\href{http://math.stanford.edu/~conrad/papers/hypercover.pdf}{math.stanford.edu/$\sim$conrad/papers/hypercover.pdf}}.

\bibitem[CT]{ct} G.~Corti\~nas and A.~Thom, ``Bivariant algebraic $K$-theory,'' {\em J.~Reine Angew.~Math.} {\bf 610} (2007), 71--123.

\bibitem[CHWW]{chww} G.~Corti\~nas, C.~Haesemeyer, M.~Walker, and C.~Weibel, ``The $K$-theory of toric varieties,'' {\em Trans.~Amer.~Math.~Soc.}~{\bf 361} (2009), 3325--3341.

\bibitem[Cu]{cuntz} J.~Cuntz, ``Bivariante $K$-Theorie f\"ur lokalkonvexe Algebren und der Chern-Connes-Charakter,'' {\em 
Doc. Math.} {\bf 2} (1997), 139--182.

\bibitem[SGA3]{sga3} M.~Demazure and A.~Grothendieck, {\em Seminaire de G\'eom\'etrie Alg\'ebrique 3, Tome I: Schemas en groupes}, Springer-Verlag, 1970.

\bibitem[Du]{dugger} D.~Dugger, ``A primer on homotopy colimits,'' notes available online at \texttt{\href{http://math.uoregon.edu/~ddugger/hocolim.pdf}{math.uoregon.edu/$\sim$ddugger/hocolim.pdf}}.

\bibitem[EG1]{eg-eit} D.~Edidin and W.~Graham, ``Equivariant intersection theory,'' {\em Invent. Math.} {\bf 131} (1998), 595--634.

\bibitem[EKMM]{ekmm} A.~Elmendorf, I.~Kriz, M.~Mandell, and J.~P.~May, {\em Rings, Modules, and Algebras in Stable Homotopy Theory}, American Mathematical Society, 1997.

\bibitem[Fr]{friedlander} E.~Friedlander, Notes from lectures at IHP in spring 2004, Lecture 6, \texttt{\href{http://www-bcf.usc.edu/~ericmf/lectures/ihp/ihplec6.pdf}{www-bcf.usc.edu/$\sim$ericmf/lectures/ihp/ihplec6.pdf}}.

\bibitem[Fu1]{Fulton} W.~Fulton, \emph{Introduction to Toric Varieties}, Princeton University Press, 1993.

\bibitem[Fu2]{it} W.~Fulton, \emph{Intersection Theory}, second edition, Springer, 1998.

\bibitem[FG]{fulton-gillet} W.~Fulton and H.~Gillet, ``Riemann-Roch for general algebraic varieties,'' \newblock{\em Bull.~Soc.~Math.~France} {\bf 111} (1983), no. 3, 287--300.

\bibitem[FL]{fl} W.~Fulton and S.~Lang, \emph{Riemann-Roch Algebra}, Springer, 1985.

\bibitem[FM]{bt} W.~Fulton and R.~MacPherson, ``Categorical framework for the study of singular spaces,'' \newblock{\em Mem.~Amer.~Math.~Soc.}~{\bf 31} (1981), no.~243. 

\bibitem[FMSS]{fmss} W.~Fulton, R.~MacPherson, F.~Sottile, and B.~Sturmfels, ``Intersection theory on spherical varieties,'' {\em J.~Algebraic Geom.} {\bf 4} (1995), no. 1, 181--193.

\bibitem[GK]{gharib-karu} S.~Gharib and K.~Karu, ``Vector bundles on toric varieties,'' {\em C.~R.~Math.~Acad.~Sci.~Paris} {\bf 350} (2012), no.~3--4, 209--212.

\bibitem[Gi]{gillet} H.~Gillet, ``Homological descent for the $K$-theory of coherent sheaves,'' in {\em Algebraic K-theory, number theory, geometry and analysis (Bielefeld, 1982)}, 80--103, Lecture Notes in Math., 1046, Springer, Berlin, 1984.

\bibitem[GKM]{gkm} M.~Goresky, R.~Kottwitz, R.~MacPherson, ``Equivariant cohomology, {K}oszul duality, and the localization theorem,'' {\em Invent. Math.} {\bf 131} (1998), no. 1, 25--83. 

\bibitem[EGA]{ega} A.~Grothendieck and J.~Dieudonn\'e, \emph{\'El\'ements de G\'eom\'etrie Alg\'ebrique III. \'Etude cohomologique des faisceaux coh\'erents. II.} Inst. Hautes \'Etudes Sci. Publ. Math. {\bf 17}, 1963.

\bibitem[Gu]{gubeladze} J.~Gubeladze, ``Toric varieties with huge Grothendieck group,'' {\em Adv.~Math.} {\bf 186} (2004), no. 1, 117--124.

\bibitem[Ha]{haesemeyer} C.~Haesemeyer, ``Descent properties of homotopy K-theory," {\em Duke Math. J.} {\bf 125} (2004), no. 3, 589--620. 

\bibitem[HHRW]{hhrw} M.~Harada, T.~Holm, N.~Ray, and G.~Williams, ``The equivariant $K$-theory and cobordism rings of divisive weighted projective spaces,'' preprint, 2013, arXiv:1306.1641v1.

\bibitem[K]{kassel} C.~Kassel, ``Caract\'ere de Chern bivariant,'' {\em $K$-theory} {\bf 3} (1989), 367--400.

\bibitem[KS]{ks} K.~Kato and T.~Saito, ``On the conductor formula of Bloch,'' {\em Publ. Math. IHES} {\bf 100} (2004), 5--151.

\bibitem[Ki]{kimura} S.-i.~Kimura, ``Fractional intersection and bivariant theory,''
{\em Comm.~Algebra} {\bf 20} (1992), no. 1, 285--302.

\bibitem[KR]{knutson-rosu} A.~Knutson and I.~Rosu, Appendix to ``Equivariant $K$-theory and equivariant cohomology,'' {\em Math. Z.} {\bf 243} (2003), no. 3, 423--448.

\bibitem[Pa1]{chow} S.~Payne, ``Equivariant Chow cohomology of toric varieties,'' {\em Math.~Res.~Lett.} {\bf 13} (2006), no. 1, 29--41.

\bibitem[Pa2]{vbs} S.~Payne, ``Toric vector bundles, branched covers of fans, and the resolution property,'' {\em J. Algebraic Geom.} {\bf 18} (2009), no. 1, 1--36.

\bibitem[Pa3]{boundary} S.~Payne, ``Boundary complexes and weight filtrations,'' {\em Mich. Math. J.} {\bf 62}, no.~2 (2013), 293--322.

\bibitem[Qu]{quillen} D.~Quillen, ``Higher algebraic $K$-theory I,'' in {\em  Algebraic $K$-theory, I: Higher $K$-theories (Proc. Conf., Battelle Memorial Inst., Seattle, Wash., 1972)}, pp. 85--147, Lecture Notes in Math., Vol.~341, Springer, Berlin 1973.

\bibitem[RW]{rw} N.~Ray and G.~Williams. In preparation.

\bibitem[Su]{sumihiro} H.~Sumihiro, ``Equivariant completion,'' {\em J. Math. Kyoto Univ.} {\bf 14} (1974), 1--28. 

\bibitem[Th1]{thomason-ss} R.~Thomason, ``First quadrant spectral sequences in algebraic $K$-theory via homotopy colimits,'' {\em Comm.~Algebra} {\bf 10} (1982), no.~15, 1589--1668.

\bibitem[Th2]{thomason} R.~Thomason, ``Algebraic $K$-theory of group scheme actions,'' in {\em Algebraic topology and algebraic $K$-theory (Princeton, N.J., 1983)}, 539--563, Ann. of Math. Stud., 113, Princeton Univ. Press, Princeton, NJ, 1987.

\bibitem[Th3]{thomason-inv} R.~Thomason, ``Lefschetz-Riemann-Roch theorem and coherent trace formula,'' {\em Invent. Math.} {\bf 85} (1986), no. 3, 515--543.

\bibitem[TT]{thomason-trobaugh} R.~Thomason and T.~Trobaugh, ``Higher algebraic $K$-theory of schemes and of derived categories, {\em The Grothendieck Festschrift, Vol.~III}, 247--435, Progr.~Math., 88, Birkh\"auser Boston, Boston, MA, 1990.

\bibitem[To1]{totaro} B.~Totaro, ``Chow groups, Chow cohomology, and linear varieties,'' {\em Forum Math. Sigma} {\bf 2} (2014), e17, 25 pp.

\bibitem[To2]{resolution} B.~Totaro, ``The resolution property for schemes and stacks,''  {\em J. Reine Angew. Math.}  {\bf 577} (2004), 1--22.

\bibitem[Tr]{traverso} C.~Traverso, ``Seminormality and Picard group," {\em Ann. Scuola Norm. Sup. Pisa (3)} {\bf 24} (1970), 585--595. 

\bibitem[VV]{vv} G.~Vezzosi and A.~Vistoli, ``Higher algebraic $K$-theory for actions of diagonalizable groups,'' {\em Invent.~Math.} {\bf 153} (2003), no. 1, 1--44.

\bibitem[We1]{weibelKH}  C.~Weibel, \emph{Homotopy algebraic $K$-theory.} Algebraic $K$-theory and algebraic number theory (Honolulu, HI, 1987), 461--488,
Contemp. Math., {\bf 83}, Amer. Math. Soc., Providence, RI, 1989.

\bibitem[We2]{weibel} C.~Weibel, {\em The $K$-book: An introduction to algebraic $K$-theory}, Graduate Studies in Mathematics, 145. American Mathematical Society, Providence, RI, 2013.

\end{thebibliography}
\end{document}